\documentclass{amsart}
\usepackage[style=alphabetic,backend=biber, sorting=nyt]{biblatex}
\bibliography{references}

\usepackage{amssymb}
\usepackage{bbm}
\usepackage{amsmath}
\usepackage{enumitem}
\usepackage{tikz-cd}
\usepackage{faktor}

\usepackage{tikz}
\usetikzlibrary{shapes.geometric}
\tikzset{
v/.style={draw, fill, circle, minimum size=1.5mm, inner sep=0},
invis/.style={draw, fill, circle, minimum size=0.3mm, inner sep=0},
b/.style={draw , regular polygon,regular polygon sides=4, minimum size=1.5mm, inner sep=.5mm},
e/.style={very thick},
vs/.style={draw, fill, circle, minimum size=1mm, inner sep=0},
bs/.style={draw,  regular polygon,regular polygon sides=4, minimum size=2mm, inner sep=0mm},
es/.style={thick}
}
\newlength{\nodeheight}
\newlength{\nodewidth}
\usetikzlibrary{arrows,matrix,positioning}

\usepackage{amsthm}
\theoremstyle{plain}
\newtheorem{theorem}{Theorem}[section]
\theoremstyle{definition} \newtheorem{definition}[theorem]{Definition}
\theoremstyle{plain}
\newtheorem{lemma}[theorem]{Lemma}
\theoremstyle{plain} \newtheorem{proposition}[theorem]{Proposition}
\theoremstyle{plain} 
\theoremstyle{plain} 
\theoremstyle{remark} 
\theoremstyle{remark} \newtheorem{remark}[theorem]{Remark}
\theoremstyle{remark} \newtheorem{example}[theorem]{Example}
\theoremstyle{remark} 
\theoremstyle{definition} 

\providecommand{\varitem}{} 


\renewcommand{\t}{\mathbbm{1}}
\newcommand{\rook}{\mathcal{R}}
\newcommand{\Tor}{\mathrm{Tor}}
\newcommand{\Brauer}{\mathrm{Br}}
\newcommand{\TL}{\mathrm{TL}}
\newcommand{\norm}[1]{\lvert{#1}\rvert}

\newcommand{\Jones}{\mathrm{J}}
\newcommand{\Partition}{\mathrm{P}}

\begin{document}

\title[Stable homology isomorphisms for the partition and Jones algebras]{Stable homology isomorphisms for the partition and Jones annular algebras}
\author{Guy Boyde}
\address{Mathematical Institute, Utrecht University, Heidelberglaan 8
3584 CS Utrecht, The Netherlands}
\email{g.boyde@uu.nl}

\subjclass[2020]{Primary 16E40, 20J06; Secondary 20B30}
\keywords{Homological stability, Partition algebras, Jones annular algebras}

\maketitle

\begin{abstract}
    We show that the homology of the Jones annular algebras is isomorphic to that of the cyclic groups below a line of gradient $\frac{1}{2}$. We also show that the homology of the partition algebras is isomorphic to that of the symmetric groups below a line of gradient 1, strengthening a result of Boyd-Hepworth-Patzt. Both isomorphisms hold in a range exceeding the stability range of the algebras in question. Along the way, we prove the usual odd-strand and invertible parameter results for the Jones annular algebras.
\end{abstract}

\section{Introduction}

To say that a family of augmented $R$-algebras $$A_1 \xrightarrow{\iota_1} A_2 \xrightarrow{\iota_2} A_3 \xrightarrow{\iota_3} A_4 \xrightarrow{\iota_4} \cdots $$
exhibits \emph{homological stability} is to say that the maps $$\Tor^{A_n}_*(\t,\t) \to \Tor^{A_{n+1}}_*(\t,\t)$$ induced on homology by $\iota_n$ are isomorphisms in a range of degrees which increases with $n$. Here $\t$ is the trivial module obtained as the quotient of $A$ by the augmentation ideal, and $R$ will be commutative and unital throughout.

Various authors have proven stability results for Temperley-Lieb algebras \cite{BH,BHComb,Sroka}, Brauer algebras \cite{BHP}, Iwahori-Hecke algebras of types $A$ and $B$ \cite{Hepworth,Moselle}, and, most recently, partition algebras \cite{BHPPartition}. In many of these cases, the authors do not prove homological stability for the algebras $A_n$ directly, but rather prove that the homology of $A_n$ is isomorphic in a range to that of some family $B_n$ whose homology is known to be stable, and then get stability for the $A_n$ as a corollary. This paper is intended to advertise the virtues of approaching such questions by \emph{fixing $n$} and attempting to resolve $B_n$ over $A_n$, forgetting altogether that $A_n$ and $B_n$ each belong to family, following the template of \cite{Sroka}.

\subsection{Results: Partition algebras}

The partition algebras $\Partition_n(\delta)$ originate independently in work of Jones \cite{JonesPartition} and Martin \cite{Martin}. Let $R$ be a (commutative unital) ring, and let $\delta \in R$. The $R$-algebra $\Partition_n(\delta)$ is the free $R$-module on the basis consisting of partitions of the set $\underline{n} \cup \underline{n}' = \{1,2,\dots, n,1',2'\dots, n'\}$, equipped with a certain multiplication rule (Definition \ref{def:partitionAlg}). Boyd, Hepworth, and Patzt \cite{BHPPartition} have proven an optimal homology stability range for the partition algebras: the map $$\Tor_q^{P_{n-1}(\delta)}(\t,\t) \to \Tor_q^{\Partition_n(\delta)}(\t,\t)$$ is an isomorphism for $q \leq \frac{n-1}{2}$. They accomplish this by showing that the homology of the partition algebras is naturally isomorphic to that of the symmetric groups in this range. Our result is as follows.

\begin{theorem} \label{thm:partition} The natural map $$\Tor_q^{\Partition_n(\delta)}(\t,\t) \to \Tor_q^{R \Sigma_n}(\t,\t) = : H_q(\Sigma_n;R)$$ is an isomorphism for $q \leq n-1$, and a surjection for $q = n$. \end{theorem}

The point is that although the partition algebras and the symmetric groups are both known to have best-possible stability ranges given by a line of gradient $\frac{1}{2}$, the isomorphism between their homologies in \cite[Theorem B]{BHPPartition} actually holds in a larger range, below a line of gradient 1.

The map appearing in the theorem statement arises as follows. A partition of $\underline{n} \cup \underline{n}'$ having precisely $n$ parts containing both a primed and an unprimed vertex can be identified with a permutation. This gives a retraction of algebras $$R \Sigma_n \to \mathrm{P}_n(\delta) \to R \Sigma_n,$$
where the second map is the quotient by the ideal $I_{\leq n-1}$ spanned by partitions not of the above form. The existence of such a group algebra retract is a general property of algebras of this flavour (see \cite[Proposition 2.3]{Boyde} for a precise statement in a slightly different context). This quotient map induces the map in homology appearing in Theorem \ref{thm:partition}.

\subsection{Results: Jones annular algebras}

The second family that we will study is the Jones annular algebras \cite{JonesAnnular} (Definition \ref{def:Jones}). We will write $\underline{n}$ for the set $\{1,2,\dots , n\}$.

Let $R$ be a ring (as always, commutative with unit), and let $\delta \in R$. Informally, the Jones annular algebra $\Jones_n(\delta)$ has a basis consisting of partitions of the set $\underline{n} \cup \underline{n}' = \{1,2,\dots, n,1',2'\dots, n'\}$ into parts of cardinality 2, such that these partitions can be represented by non-intersecting edges when the vertices are embedded in the ends of the cylinder. In other words, it is the `cylindrical' or `annular' version of the Temperley-Lieb algebra defined in \cite{TemperleyLieb} and studied in \cite{BH} and \cite{Sroka}.

Jones annular algebras are hence also closely related to the Brauer algebras $\Brauer_n(\delta)$ defined in \cite{Brauer} and studied in \cite{BHP}, and the partition algebras above. Precisely, there are inclusions:
$$\TL_n(\delta) \subset \Jones_n(\delta) \subset \Brauer_n(\delta) \subset \Partition_n(\delta).$$
For completeness, we remark that the rook-Brauer algebras $\rook \Brauer_n(\delta, \varepsilon)$ defined in \cite{HalversondelMas} and \cite{MartinMazorchuk} and studied in \cite{Boyde} contain the Brauer algebras and are contained in the Partition algebras when $\varepsilon = \delta$, but will not appear in this paper. All of these algebras are \emph{cellular} in the sense of Graham and Lehrer \cite{GrahamLehrer,Xi,MartinMazorchuk}, but we will not make much explicit use of this.

There is one important subtlety in the definition of $\Jones_n(\delta)$: pictorial representatives which differ `by a Dehn twist' correspond to the same basis element. That is, unlike the Temperley-Lieb algebra, a single basis element (i.e. a single pairing on the vertices) can be represented by multiple non-isotopic pictures. If this were not so then $\Jones_n(\delta)$ would be infinite dimensional, and would not be a subalgebra of the partition algebra. This makes it a slight abuse of terminology to call the basis elements \emph{annular diagrams}, but we will do so nonetheless.

\begin{example}
Here are two non-isotopic pictorial representatives of a single basis element in $J_{11}(\delta)$ - they differ only by a Dehn twist. The connections are coloured only to more clearly indicate which is which - this has no meaning in the algebra. The cylinder is drawn as a pasting diagram (the dotted lines to be identified).

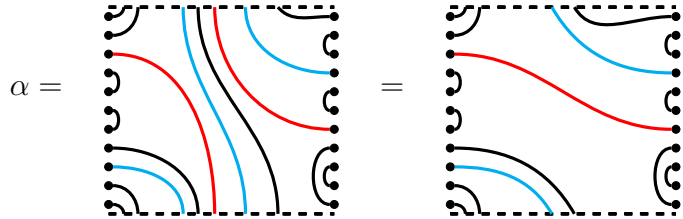
\begin{center}
$\alpha = $
\quad
\begin{tikzpicture}[x=1.5cm,y=-.5cm,baseline=-1.05cm]

\def\wid{2}
\def\cheatWid{\wid*1.05}
\def\hei{0.5}
\def\nodesize{3}
\def\ang{90}

\draw[e, dashed] (\wid,-0.5*\hei) rectangle (2*\wid,-0.5*\hei);
\draw[e, dashed] (\wid,10.5*\hei) rectangle (2*\wid,10.5*\hei);

\node[v, minimum size=\nodesize] (b1) at (1* \wid,0*\hei) {};
\node[v, minimum size=\nodesize] (b2) at (1* \wid,1*\hei) {};
\node[v, minimum size=\nodesize] (b3) at (1* \wid,2*\hei) {};
\node[v, minimum size=\nodesize] (b4) at (1* \wid,3*\hei) {};
\node[v, minimum size=\nodesize] (b5) at (1* \wid,4*\hei) {};
\node[v, minimum size=\nodesize] (b6) at (1* \wid,5*\hei) {};
\node[v, minimum size=\nodesize] (b7) at (1* \wid,6*\hei) {};
\node[v, minimum size=\nodesize] (b8) at (1* \wid,7*\hei) {};
\node[v, minimum size=\nodesize] (b9) at (1* \wid,8*\hei) {};
\node[v, minimum size=\nodesize] (b10) at (1* \wid,9*\hei) {};
\node[v, minimum size=\nodesize] (b11) at (1* \wid,10*\hei) {};

\node[v, minimum size=\nodesize] (c1) at (2* \wid,0*\hei) {};
\node[v, minimum size=\nodesize] (c2) at (2* \wid,1*\hei) {};
\node[v, minimum size=\nodesize] (c3) at (2* \wid,2*\hei) {};
\node[v, minimum size=\nodesize] (c4) at (2* \wid,3*\hei) {};
\node[v, minimum size=\nodesize] (c5) at (2* \wid,4*\hei) {};
\node[v, minimum size=\nodesize] (c6) at (2* \wid,5*\hei) {};
\node[v, minimum size=\nodesize] (c7) at (2* \wid,6*\hei) {};
\node[v, minimum size=\nodesize] (c8) at (2* \wid,7*\hei) {};
\node[v, minimum size=\nodesize] (c9) at (2* \wid,8*\hei) {};
\node[v, minimum size=\nodesize] (c10) at (2* \wid,9*\hei) {};
\node[v, minimum size=\nodesize] (c11) at (2* \wid,10*\hei) {};

\draw[e] (b1) to[out = 0, in = 270] (1.067 * \wid,-0.5*\hei);
\draw[e] (b2) to[out = 0, in = 270] (1.13 * \wid,-0.5*\hei);
\draw[e] (b11) to[out = 0, in = 90] (1.067 * \wid,10.5*\hei);
\draw[e] (b10) to[out = 0, in = 90] (1.13 * \wid,10.5*\hei);

\draw[e] (b4) to[out=0, in=0] (b5);
\draw[e] (b6) to[out=0, in=0] (b7);

\draw[e,red] (b3) to[out=0, in=\ang] (1.4 * \cheatWid,10.5*\hei);
\draw[e] (b8) to[out=0, in=\ang] (1.33 * \cheatWid,10.5*\hei);
\draw[e,cyan] (b9) to[out=0, in=\ang] (1.267 * \cheatWid,10.5*\hei);

\draw[e,red] (1.4 * \cheatWid,-0.5*\hei) to[out=180+\ang, in=180] (c7);
\draw[e] (1.33 * \cheatWid,-0.5*\hei) to[out=180+\ang, in=\ang] (1.667 * \cheatWid,10.5*\hei);
\draw[e,cyan] (1.267 * \cheatWid,-0.5*\hei) to[out=180+\ang, in=\ang] (1.53 * \cheatWid,10.5*\hei);

\draw[e] (1.667 * \cheatWid,-0.5*\hei) to[out=220+\ang, in=180] (c1);
\draw[e,cyan] (1.53 * \cheatWid,-0.5*\hei) to[out=180+\ang, in=180] (c4);

\draw[e] (c5) to[out=180, in=180] (c6);
\draw[e] (c2) to[out=180, in=180] (c3);
\draw[e] (c8) to[out=180, in=180] (c11);
\draw[e] (c9) to[out=180, in=180] (c10);

\end{tikzpicture}
\quad
$=$
\quad
\begin{tikzpicture}[x=1.5cm,y=-.5cm,baseline=-1.05cm]

\def\wid{2}
\def\cheatWid{\wid*1.05}
\def\hei{0.5}
\def\nodesize{3}
\def\newAng{120}

\draw[e, dashed] (\wid,-0.5*\hei) rectangle (2*\wid,-0.5*\hei);
\draw[e, dashed] (\wid,10.5*\hei) rectangle (2*\wid,10.5*\hei);

\node[v, minimum size=\nodesize] (b1) at (1* \wid,0*\hei) {};
\node[v, minimum size=\nodesize] (b2) at (1* \wid,1*\hei) {};
\node[v, minimum size=\nodesize] (b3) at (1* \wid,2*\hei) {};
\node[v, minimum size=\nodesize] (b4) at (1* \wid,3*\hei) {};
\node[v, minimum size=\nodesize] (b5) at (1* \wid,4*\hei) {};
\node[v, minimum size=\nodesize] (b6) at (1* \wid,5*\hei) {};
\node[v, minimum size=\nodesize] (b7) at (1* \wid,6*\hei) {};
\node[v, minimum size=\nodesize] (b8) at (1* \wid,7*\hei) {};
\node[v, minimum size=\nodesize] (b9) at (1* \wid,8*\hei) {};
\node[v, minimum size=\nodesize] (b10) at (1* \wid,9*\hei) {};
\node[v, minimum size=\nodesize] (b11) at (1* \wid,10*\hei) {};

\node[v, minimum size=\nodesize] (c1) at (2* \wid,0*\hei) {};
\node[v, minimum size=\nodesize] (c2) at (2* \wid,1*\hei) {};
\node[v, minimum size=\nodesize] (c3) at (2* \wid,2*\hei) {};
\node[v, minimum size=\nodesize] (c4) at (2* \wid,3*\hei) {};
\node[v, minimum size=\nodesize] (c5) at (2* \wid,4*\hei) {};
\node[v, minimum size=\nodesize] (c6) at (2* \wid,5*\hei) {};
\node[v, minimum size=\nodesize] (c7) at (2* \wid,6*\hei) {};
\node[v, minimum size=\nodesize] (c8) at (2* \wid,7*\hei) {};
\node[v, minimum size=\nodesize] (c9) at (2* \wid,8*\hei) {};
\node[v, minimum size=\nodesize] (c10) at (2* \wid,9*\hei) {};
\node[v, minimum size=\nodesize] (c11) at (2* \wid,10*\hei) {};

\draw[e] (b1) to[out = 0, in = 270] (1.067 * \wid,-0.5*\hei);
\draw[e] (b2) to[out = 0, in = 270] (1.13 * \wid,-0.5*\hei);
\draw[e] (b11) to[out = 0, in = 90] (1.067 * \wid,10.5*\hei);
\draw[e] (b10) to[out = 0, in = 90] (1.13 * \wid,10.5*\hei);

\draw[e] (b4) to[out=0, in=0] (b5);
\draw[e] (b6) to[out=0, in=0] (b7);

\draw[e,red] (b3) to[out=0, in=180] (c7);
\draw[e] (b8) to[out=0, in=\newAng] (1.55 * \wid,10.5*\hei);
\draw[e,cyan] (b9) to[out=0, in=\newAng] (1.45 * \wid,10.5*\hei);

\draw[e] (1.55 * \wid,-0.5*\hei) to[out=180+\newAng, in=180] (c1);
\draw[e,cyan] (1.45 * \wid,-0.5*\hei) to[out=180+\newAng, in=180] (c4);

\draw[e] (c5) to[out=180, in=180] (c6);
\draw[e] (c2) to[out=180, in=180] (c3);
\draw[e] (c8) to[out=180, in=180] (c11);
\draw[e] (c9) to[out=180, in=180] (c10);

\end{tikzpicture}
\quad
\end{center}
\end{example}

\begin{example} Here is a sample multiplication in $\mathrm{J}_{11}(\delta)$. Let $\alpha$ be the (right) diagram from the previous example, and let

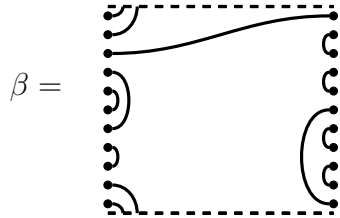
\begin{center}
$\beta = $
\quad
\begin{tikzpicture}[x=1.5cm,y=-.5cm,baseline=-1.05cm]

\def\wid{2}
\def\cheatWid{\wid*1.05}
\def\hei{0.5}
\def\nodesize{3}
\def\newAng{120}

\draw[e, dashed] (2*\wid,-0.5*\hei) rectangle (3*\wid,-0.5*\hei);
\draw[e, dashed] (2*\wid,10.5*\hei) rectangle (3*\wid,10.5*\hei);

\node[v, minimum size=\nodesize] (c1) at (2* \wid,0*\hei) {};
\node[v, minimum size=\nodesize] (c2) at (2* \wid,1*\hei) {};
\node[v, minimum size=\nodesize] (c3) at (2* \wid,2*\hei) {};
\node[v, minimum size=\nodesize] (c4) at (2* \wid,3*\hei) {};
\node[v, minimum size=\nodesize] (c5) at (2* \wid,4*\hei) {};
\node[v, minimum size=\nodesize] (c6) at (2* \wid,5*\hei) {};
\node[v, minimum size=\nodesize] (c7) at (2* \wid,6*\hei) {};
\node[v, minimum size=\nodesize] (c8) at (2* \wid,7*\hei) {};
\node[v, minimum size=\nodesize] (c9) at (2* \wid,8*\hei) {};
\node[v, minimum size=\nodesize] (c10) at (2* \wid,9*\hei) {};
\node[v, minimum size=\nodesize] (c11) at (2* \wid,10*\hei) {};

\node[v, minimum size=\nodesize] (d1) at (3* \wid,0*\hei) {};
\node[v, minimum size=\nodesize] (d2) at (3* \wid,1*\hei) {};
\node[v, minimum size=\nodesize] (d3) at (3* \wid,2*\hei) {};
\node[v, minimum size=\nodesize] (d4) at (3* \wid,3*\hei) {};
\node[v, minimum size=\nodesize] (d5) at (3* \wid,4*\hei) {};
\node[v, minimum size=\nodesize] (d6) at (3* \wid,5*\hei) {};
\node[v, minimum size=\nodesize] (d7) at (3* \wid,6*\hei) {};
\node[v, minimum size=\nodesize] (d8) at (3* \wid,7*\hei) {};
\node[v, minimum size=\nodesize] (d9) at (3* \wid,8*\hei) {};
\node[v, minimum size=\nodesize] (d10) at (3* \wid,9*\hei) {};
\node[v, minimum size=\nodesize] (d11) at (3* \wid,10*\hei) {};

\draw[e] (c1) to[out = 0, in = 270] (2.067 * \wid,-0.5*\hei);
\draw[e] (c2) to[out = 0, in = 270] (2.13 * \wid,-0.5*\hei);
\draw[e] (c11) to[out = 0, in = 90] (2.067 * \wid,10.5*\hei);
\draw[e] (c10) to[out = 0, in = 90] (2.13 * \wid,10.5*\hei);

\draw[e] (c4) to[out=0, in=0] (c7);
\draw[e] (c6) to[out=0, in=0] (c5);
\draw[e] (c8) to[out=0, in=0] (c9);

\draw[e] (c3) to[out=0, in=180] (d1);

\draw[e] (d4) to[out=180, in=180] (d5);
\draw[e] (d2) to[out=180, in=180] (d3);
\draw[e] (d6) to[out=180, in=180] (d11);
\draw[e] (d7) to[out=180, in=180] (d8);
\draw[e] (d9) to[out=180, in=180] (d10);

\end{tikzpicture}
\quad
\end{center}
Then $\alpha \beta$ is computed by forming the `composed diagram'
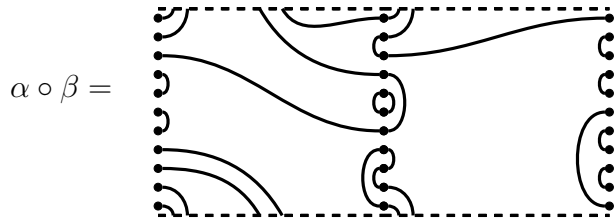
\begin{center}
$\alpha \circ \beta = $
\quad
\begin{tikzpicture}[x=1.5cm,y=-.5cm,baseline=-1.05cm]

\def\wid{2}
\def\cheatWid{\wid*1.05}
\def\hei{0.5}
\def\nodesize{3}
\def\newAng{120}


\draw[e, dashed] (\wid,-0.5*\hei) rectangle (2*\wid,-0.5*\hei);
\draw[e, dashed] (\wid,10.5*\hei) rectangle (2*\wid,10.5*\hei);

\node[v, minimum size=\nodesize] (b1) at (1* \wid,0*\hei) {};
\node[v, minimum size=\nodesize] (b2) at (1* \wid,1*\hei) {};
\node[v, minimum size=\nodesize] (b3) at (1* \wid,2*\hei) {};
\node[v, minimum size=\nodesize] (b4) at (1* \wid,3*\hei) {};
\node[v, minimum size=\nodesize] (b5) at (1* \wid,4*\hei) {};
\node[v, minimum size=\nodesize] (b6) at (1* \wid,5*\hei) {};
\node[v, minimum size=\nodesize] (b7) at (1* \wid,6*\hei) {};
\node[v, minimum size=\nodesize] (b8) at (1* \wid,7*\hei) {};
\node[v, minimum size=\nodesize] (b9) at (1* \wid,8*\hei) {};
\node[v, minimum size=\nodesize] (b10) at (1* \wid,9*\hei) {};
\node[v, minimum size=\nodesize] (b11) at (1* \wid,10*\hei) {};

\node[v, minimum size=\nodesize] (c1) at (2* \wid,0*\hei) {};
\node[v, minimum size=\nodesize] (c2) at (2* \wid,1*\hei) {};
\node[v, minimum size=\nodesize] (c3) at (2* \wid,2*\hei) {};
\node[v, minimum size=\nodesize] (c4) at (2* \wid,3*\hei) {};
\node[v, minimum size=\nodesize] (c5) at (2* \wid,4*\hei) {};
\node[v, minimum size=\nodesize] (c6) at (2* \wid,5*\hei) {};
\node[v, minimum size=\nodesize] (c7) at (2* \wid,6*\hei) {};
\node[v, minimum size=\nodesize] (c8) at (2* \wid,7*\hei) {};
\node[v, minimum size=\nodesize] (c9) at (2* \wid,8*\hei) {};
\node[v, minimum size=\nodesize] (c10) at (2* \wid,9*\hei) {};
\node[v, minimum size=\nodesize] (c11) at (2* \wid,10*\hei) {};

\draw[e] (b1) to[out = 0, in = 270] (1.067 * \wid,-0.5*\hei);
\draw[e] (b2) to[out = 0, in = 270] (1.13 * \wid,-0.5*\hei);
\draw[e] (b11) to[out = 0, in = 90] (1.067 * \wid,10.5*\hei);
\draw[e] (b10) to[out = 0, in = 90] (1.13 * \wid,10.5*\hei);

\draw[e] (b4) to[out=0, in=0] (b5);
\draw[e] (b6) to[out=0, in=0] (b7);

\draw[e] (b3) to[out=0, in=180] (c7);
\draw[e] (b8) to[out=0, in=\newAng] (1.55 * \wid,10.5*\hei);
\draw[e] (b9) to[out=0, in=\newAng] (1.45 * \wid,10.5*\hei);

\draw[e] (1.55 * \wid,-0.5*\hei) to[out=180+\newAng, in=180] (c1);
\draw[e] (1.45 * \wid,-0.5*\hei) to[out=180+\newAng, in=180] (c4);

\draw[e] (c5) to[out=180, in=180] (c6);
\draw[e] (c2) to[out=180, in=180] (c3);
\draw[e] (c8) to[out=180, in=180] (c11);
\draw[e] (c9) to[out=180, in=180] (c10);


\draw[e, dashed] (2*\wid,-0.5*\hei) rectangle (3*\wid,-0.5*\hei);
\draw[e, dashed] (2*\wid,10.5*\hei) rectangle (3*\wid,10.5*\hei);

\node[v, minimum size=\nodesize] (c1) at (2* \wid,0*\hei) {};
\node[v, minimum size=\nodesize] (c2) at (2* \wid,1*\hei) {};
\node[v, minimum size=\nodesize] (c3) at (2* \wid,2*\hei) {};
\node[v, minimum size=\nodesize] (c4) at (2* \wid,3*\hei) {};
\node[v, minimum size=\nodesize] (c5) at (2* \wid,4*\hei) {};
\node[v, minimum size=\nodesize] (c6) at (2* \wid,5*\hei) {};
\node[v, minimum size=\nodesize] (c7) at (2* \wid,6*\hei) {};
\node[v, minimum size=\nodesize] (c8) at (2* \wid,7*\hei) {};
\node[v, minimum size=\nodesize] (c9) at (2* \wid,8*\hei) {};
\node[v, minimum size=\nodesize] (c10) at (2* \wid,9*\hei) {};
\node[v, minimum size=\nodesize] (c11) at (2* \wid,10*\hei) {};

\node[v, minimum size=\nodesize] (d1) at (3* \wid,0*\hei) {};
\node[v, minimum size=\nodesize] (d2) at (3* \wid,1*\hei) {};
\node[v, minimum size=\nodesize] (d3) at (3* \wid,2*\hei) {};
\node[v, minimum size=\nodesize] (d4) at (3* \wid,3*\hei) {};
\node[v, minimum size=\nodesize] (d5) at (3* \wid,4*\hei) {};
\node[v, minimum size=\nodesize] (d6) at (3* \wid,5*\hei) {};
\node[v, minimum size=\nodesize] (d7) at (3* \wid,6*\hei) {};
\node[v, minimum size=\nodesize] (d8) at (3* \wid,7*\hei) {};
\node[v, minimum size=\nodesize] (d9) at (3* \wid,8*\hei) {};
\node[v, minimum size=\nodesize] (d10) at (3* \wid,9*\hei) {};
\node[v, minimum size=\nodesize] (d11) at (3* \wid,10*\hei) {};

\draw[e] (c1) to[out = 0, in = 270] (2.067 * \wid,-0.5*\hei);
\draw[e] (c2) to[out = 0, in = 270] (2.13 * \wid,-0.5*\hei);
\draw[e] (c11) to[out = 0, in = 90] (2.067 * \wid,10.5*\hei);
\draw[e] (c10) to[out = 0, in = 90] (2.13 * \wid,10.5*\hei);

\draw[e] (c4) to[out=0, in=0] (c7);
\draw[e] (c6) to[out=0, in=0] (c5);
\draw[e] (c8) to[out=0, in=0] (c9);

\draw[e] (c3) to[out=0, in=180] (d1);

\draw[e] (d4) to[out=180, in=180] (d5);
\draw[e] (d2) to[out=180, in=180] (d3);
\draw[e] (d6) to[out=180, in=180] (d11);
\draw[e] (d7) to[out=180, in=180] (d8);
\draw[e] (d9) to[out=180, in=180] (d10);

\end{tikzpicture}
\quad
\end{center}
and (just like in the Temperley-Lieb algebra) replacing each loop appearing in the middle with a factor of $\delta \in R$. This gives:

\begin{center}
$\alpha \beta = \delta \cdot$
\quad
\begin{tikzpicture}[x=1.5cm,y=-.5cm,baseline=-1.05cm]

\def\wid{2}
\def\cheatWid{\wid*1.05}
\def\hei{0.5}
\def\nodesize{3}
\def\newAng{120}

\draw[e, dashed] (\wid,-0.5*\hei) rectangle (2*\wid,-0.5*\hei);
\draw[e, dashed] (\wid,10.5*\hei) rectangle (2*\wid,10.5*\hei);

\node[v, minimum size=\nodesize] (b1) at (1* \wid,0*\hei) {};
\node[v, minimum size=\nodesize] (b2) at (1* \wid,1*\hei) {};
\node[v, minimum size=\nodesize] (b3) at (1* \wid,2*\hei) {};
\node[v, minimum size=\nodesize] (b4) at (1* \wid,3*\hei) {};
\node[v, minimum size=\nodesize] (b5) at (1* \wid,4*\hei) {};
\node[v, minimum size=\nodesize] (b6) at (1* \wid,5*\hei) {};
\node[v, minimum size=\nodesize] (b7) at (1* \wid,6*\hei) {};
\node[v, minimum size=\nodesize] (b8) at (1* \wid,7*\hei) {};
\node[v, minimum size=\nodesize] (b9) at (1* \wid,8*\hei) {};
\node[v, minimum size=\nodesize] (b10) at (1* \wid,9*\hei) {};
\node[v, minimum size=\nodesize] (b11) at (1* \wid,10*\hei) {};

\node[invis] (c1) at (1.5* \wid,0*\hei) {};
\node[invis] (c2) at (1.5* \wid,1*\hei) {};
\node[invis] (c3) at (1.5* \wid,2*\hei) {};
\node[invis] (c4) at (1.5* \wid,3*\hei) {};
\node[invis] (c7) at (1.5* \wid,6*\hei) {};
\node[invis] (c8) at (1.5* \wid,7*\hei) {};
\node[invis] (c9) at (1.5* \wid,8*\hei) {};
\node[invis] (c10) at (1.5* \wid,9*\hei) {};
\node[invis] (c11) at (1.5* \wid,10*\hei) {};

\node[v, minimum size=\nodesize] (d1) at (2* \wid,0*\hei) {};
\node[v, minimum size=\nodesize] (d2) at (2* \wid,1*\hei) {};
\node[v, minimum size=\nodesize] (d3) at (2* \wid,2*\hei) {};
\node[v, minimum size=\nodesize] (d4) at (2* \wid,3*\hei) {};
\node[v, minimum size=\nodesize] (d5) at (2* \wid,4*\hei) {};
\node[v, minimum size=\nodesize] (d6) at (2* \wid,5*\hei) {};
\node[v, minimum size=\nodesize] (d7) at (2* \wid,6*\hei) {};
\node[v, minimum size=\nodesize] (d8) at (2* \wid,7*\hei) {};
\node[v, minimum size=\nodesize] (d9) at (2* \wid,8*\hei) {};
\node[v, minimum size=\nodesize] (d10) at (2* \wid,9*\hei) {};
\node[v, minimum size=\nodesize] (d11) at (2* \wid,10*\hei) {};

\draw[e] (b1) to[out = 0, in = 270] (1.067 * \wid,-0.5*\hei);
\draw[e] (b2) to[out = 0, in = 270] (1.13 * \wid,-0.5*\hei);
\draw[e] (b11) to[out = 0, in = 90] (1.067 * \wid,10.5*\hei);
\draw[e] (b10) to[out = 0, in = 90] (1.13 * \wid,10.5*\hei);

\draw[e] (c1) to[out = 0, in = 270] (1.567 * \wid,-0.5*\hei);
\draw[e] (c2) to[out = 0, in = 270] (1.63 * \wid,-0.5*\hei);
\draw[e] (c11) to[out = 0, in = 90] (1.567 * \wid,10.5*\hei);
\draw[e] (c10) to[out = 0, in = 90] (1.63 * \wid,10.5*\hei);

\draw[e] (b4) to[out=0, in=0] (b5);
\draw[e] (b6) to[out=0, in=0] (b7);

\draw[e] (b8) to[out=0, in=\newAng] (1.275 * \wid,10.5*\hei);
\draw[e] (b9) to[out=0, in=\newAng] (1.225 * \wid,10.5*\hei);

\draw[e] (1.275 * \wid,-0.5*\hei) to[out=180+\newAng, in=180] (c1);
\draw[e] (1.225 * \wid,-0.5*\hei) to[out=180+\newAng, in=180] (c4);

\draw[e] (b3) to[out=0, in=180] (c7);

\draw[e] (c2) to[out=180, in=180] (c3);
\draw[e] (c4) to[out=0, in=0] (c7);
\draw[e] (c8) to[out=0, in=0] (c9);

\draw[e] (c3) to[out=0, in=180] (d1);

\draw[e] (c8) to[out=180, in=180] (c11);
\draw[e] (c9) to[out=180, in=180] (c10);

\draw[e] (d4) to[out=180, in=180] (d5);
\draw[e] (d2) to[out=180, in=180] (d3);
\draw[e] (d6) to[out=180, in=180] (d11);
\draw[e] (d7) to[out=180, in=180] (d8);
\draw[e] (d9) to[out=180, in=180] (d10);

\end{tikzpicture}
\quad
$= \delta \cdot$
\quad
\begin{tikzpicture}[x=1.5cm,y=-.5cm,baseline=-1.05cm]

\def\wid{2}
\def\cheatWid{\wid*1.05}
\def\hei{0.5}
\def\nodesize{3}
\def\newAng{120}

\draw[e, dashed] (\wid,-0.5*\hei) rectangle (2*\wid,-0.5*\hei);
\draw[e, dashed] (\wid,10.5*\hei) rectangle (2*\wid,10.5*\hei);

\node[v, minimum size=\nodesize] (b1) at (1* \wid,0*\hei) {};
\node[v, minimum size=\nodesize] (b2) at (1* \wid,1*\hei) {};
\node[v, minimum size=\nodesize] (b3) at (1* \wid,2*\hei) {};
\node[v, minimum size=\nodesize] (b4) at (1* \wid,3*\hei) {};
\node[v, minimum size=\nodesize] (b5) at (1* \wid,4*\hei) {};
\node[v, minimum size=\nodesize] (b6) at (1* \wid,5*\hei) {};
\node[v, minimum size=\nodesize] (b7) at (1* \wid,6*\hei) {};
\node[v, minimum size=\nodesize] (b8) at (1* \wid,7*\hei) {};
\node[v, minimum size=\nodesize] (b9) at (1* \wid,8*\hei) {};
\node[v, minimum size=\nodesize] (b10) at (1* \wid,9*\hei) {};
\node[v, minimum size=\nodesize] (b11) at (1* \wid,10*\hei) {};

\node[v, minimum size=\nodesize] (d1) at (2* \wid,0*\hei) {};
\node[v, minimum size=\nodesize] (d2) at (2* \wid,1*\hei) {};
\node[v, minimum size=\nodesize] (d3) at (2* \wid,2*\hei) {};
\node[v, minimum size=\nodesize] (d4) at (2* \wid,3*\hei) {};
\node[v, minimum size=\nodesize] (d5) at (2* \wid,4*\hei) {};
\node[v, minimum size=\nodesize] (d6) at (2* \wid,5*\hei) {};
\node[v, minimum size=\nodesize] (d7) at (2* \wid,6*\hei) {};
\node[v, minimum size=\nodesize] (d8) at (2* \wid,7*\hei) {};
\node[v, minimum size=\nodesize] (d9) at (2* \wid,8*\hei) {};
\node[v, minimum size=\nodesize] (d10) at (2* \wid,9*\hei) {};
\node[v, minimum size=\nodesize] (d11) at (2* \wid,10*\hei) {};

\draw[e] (b1) to[out = 0, in = 270] (1.067 * \wid,-0.5*\hei);
\draw[e] (b2) to[out = 0, in = 270] (1.13 * \wid,-0.5*\hei);
\draw[e] (b3) to[out = 0, in = 270] (1.2 * \wid,-0.5*\hei);
\draw[e] (b11) to[out = 0, in = 90] (1.067 * \wid,10.5*\hei);
\draw[e] (b10) to[out = 0, in = 90] (1.13 * \wid,10.5*\hei);
\draw[e] (b9) to[out = 0, in = 90] (1.2 * \wid,10.5*\hei);

\draw[e] (b4) to[out=0, in=0] (b5);
\draw[e] (b6) to[out=0, in=0] (b7);

\draw[e] (b8) to[out=0, in=180] (d1);

\draw[e] (d4) to[out=180, in=180] (d5);
\draw[e] (d2) to[out=180, in=180] (d3);
\draw[e] (d6) to[out=180, in=180] (d11);
\draw[e] (d7) to[out=180, in=180] (d8);
\draw[e] (d9) to[out=180, in=180] (d10);

\end{tikzpicture}
\quad
$\in J_{11}(\delta),$
\end{center} where we simplify the result as in the Temperley-Lieb algebra.
\end{example}

An annular diagram having $n$ left-to-right connections may be identified with an element of the cyclic group $C_n$. As for the partition algebras, this gives a retraction $$R C_n \to \Jones_n(\delta) \to R C_n$$
onto a group algebra, where the second map is the quotient by the ideal $I_{\leq n-1}$ spanned by diagrams with fewer than the maximal number $n$ of left-to-right connections. We note that, unlike the Partition algebras, the Jones algebras fit into the setting of \cite{Boyde}: the subalgebra $R C_n \subset \Jones_n(\delta)$ is exactly the `subalgebra on diagrams with the maximal number of left-to-right connections' appearing in \cite[Proposition 2.3]{Boyde}.

Our main result for the Jones annular algebras is then as follows.

\begin{theorem} \label{thm:Jones} The natural map $$\Tor_q^{\Jones_n(\delta)}(\t,\t) \to \Tor_q^{R C_n}(\t,\t) = : H_q(C_n;R)$$ is an isomorphism for $q \leq \frac{n}{2}-1$, and a surjection for $q = \frac{n}{2}$.
\end{theorem}

The cyclic groups (hence also the Jones annular algebras) cannot exhibit any kind of homological stability (for example $H_1(C_n;\mathbb{Z}) \cong \mathbb{Z}/n$), but the two families are nonetheless stably isomorphic.

We also prove the following.

\begin{theorem} \label{thm:JonesSroka} If $n$ is odd or $\delta$ is invertible, then the map $$\Tor_q^{\Jones_n(\delta)}(\t,\t) \to H_q(C_n;R)$$ of Theorem \ref{thm:Jones} is an isomorphism for all $q$.
\end{theorem}

Theorems \ref{thm:Jones} and \ref{thm:JonesSroka} are analogues for the Jones annular algebras of Sroka's results on the Temperley-Lieb algebras \cite{Sroka}. Theorem \ref{thm:JonesSroka} can also be deduced using the author's earlier result \cite[Theorem 1.7]{Boyde}, and something like this is done (in slightly more representation-theoretic language) in a sequel to this paper \cite[Corollary 8.5]{BoydeCellular}.

\subsection{Methods}

Each of our two main theorems asserts that a certain map of algebras $A \to B$ is a homology isomorphism in a range. What we are doing in each case is
\begin{itemize}
    \item constructing a partial projective resolution of $B$ as an $A$-module, and
    \item using this as input to a change of rings spectral sequence.
\end{itemize}

In practice, we will make (the usual) more elementary argument (see the proof of Theorem \ref{thm:MN}) instead of using the change of rings spectral sequence. The range of degrees in which we succeed in constructing the resolution will become our stability range.

For a positive integer $w$ we will write $\underline{w}$ for the set $\{1, \dots , w\}$.

\begin{definition} \label{def:IdempotentCover} Let $A$ be an $R$-algebra, let $I$ be a twosided ideal of $A$, and let $w \geq h \geq 1$. An \emph{idempotent (left) cover} of $I$ of \emph{height $h$} and \emph{width $w$} is a finite collection of left ideals $J_1, \dots , J_w$ of $A$, which cover $I$ in the sense that $J_1 + \dots + J_w = I$, and such that for each $S \subset \underline{w}$ with $\norm{S} \leq h$, the intersection $$\bigcap_{i \in S} J_i$$ is either zero or is a principal left ideal generated by an idempotent. If $I$ is free as an $R$-module, then an idempotent cover is said to be \emph{$R$-free} if there is a choice of $R$-basis for $I$ such that each $J_i$ is free on a subset of this basis.
\end{definition}

Our main technical theorem is as follows.

\begin{theorem} \label{thm:main} Let $A$ be an augmented $R$-algebra with trivial module $\t$. Let $I$ be a twosided ideal of $A$ which is free as an $R$-module and acts trivially on $\t$. Suppose that there exists an $R$-free idempotent left cover of $I$ of height $h$. Then the natural map $$\Tor_q^A(\t,\t) \to \Tor_q^{\faktor{A}{I}}(\t,\t)$$ is an isomorphism for $q \leq h$, and a surjection for $q = h+1$. If $h$ is equal to the width $w$ of the cover, then this map is an isomorphism for all $q$.
\end{theorem}

This theorem will be proven in Section \ref{subs:StableIso}; it is essentially the combination of Theorem \ref{thm:MN} and Proposition \ref{prop:MVisEllProjRes} in the case $M=N=\t$. We should emphasise that this theorem is fairly elementary, and almost certainly already known in some guise. It nonetheless seems to be quite an effective tool.



The idea is as follows. Given a cover as in Definition \ref{def:IdempotentCover}, one can always form what we will call a \emph{Mayer-Vietoris complex} of $A$-modules $C_*$ (Subsection \ref{subsection: MV}), where the module $C_p$ in degree $p$ is the direct sum of the $p$-fold intersections of the ideals of the cover. The prototype is Sroka's \emph{cellular Davis complex} \cite{Sroka}, which in our language is the Mayer-Vietoris complex associated to a certain cover of the augmentation ideal of a Temperley-Lieb algebra, with quotient $R$.

More precisely, this Mayer-Vietoris complex $C_*$ of $\faktor{A}{I}$ over $A$ is a chain complex of $A$-modules with the following three properties (Proposition \ref{prop:MVisEllProjRes}):
\begin{itemize}
    \item The complex $C_*$ is exact provided that the cover was free (Lemma \ref{lem:MVac}),
    \item the terms $C_p$ are projective in degrees $p \leq h$, since principal ideals generated by idempotents are in particular projective as modules, and
    \item in the same range the coinvariants $X \otimes_A C_p$ vanish for any right $A$-module $X$ on which $I$ acts trivially (Lemma \ref{lem:idempotentTensor}).
\end{itemize}
Once one has these properties, Theorem \ref{thm:main} follows by abstract homological algebra. The key point is that the height $h$ of the cover becomes the range in which we can construct an explicit partial projective resolution, and hence the range in which we ultimately get an isomorphism.

From this point of view, the reason for the `$n$ odd results' (Theorem \ref{thm:JonesSroka}, as well as the $n$ odd results of \cite{Sroka} and \cite{Boyde}) is that in those situations the relevant Mayer-Vietoris complex is a genuine projective resolution (that is, $h=w$ in Theorem \ref{thm:main}). When $n$ is even, however, the highest-degree term of the complex fails to be projective, and the range in which we get an identification of the homology depends on the degree of this term.

The style of abstraction appearing in Definition \ref{def:IdempotentCover} and Theorem \ref{thm:main} is similar to that of the author's earlier paper \cite{Boyde} (especially Section 3 of that paper). In particular, the proof of Theorem \ref{thm:main} uses essentially the same algebra as that of \cite[Theorem 3.3]{Boyde}, but packaged instead in a way that abstracts Sroka's approach \cite{Sroka}. This makes it possible to prove genuine stability results that hold in a range of degrees depending on $h$. The manner of organising the data in \cite{Boyde} meant that one could not `get hold of' the height variable $h$ which is so key here. The reason for this is that \cite{Boyde} insists on working with covers by direct sums of ideals (i.e. where even the twofold intersections are trivial).

In a sequel \cite{BoydeCellular}, we will show that some of what we do here can be reframed in a more representation-theoretic way. The advantages of this reframing seem to be mainly conceptual - in particular, the methods we use here seem to be at least as efficient. We therefore think of this paper as using topological methods to get results, and \cite{BoydeCellular} as exploring the extent to which the statements of those results can be made representation-theoretic.

\subsection*{Acknowledgements}

Special thanks are due to an anonymous referee for many helpful comments, which have greatly improved the clarity of the exposition, simplified several proofs (especially that of Proposition \ref{prop:SpecSeq}), and improved the ranges in the main theorems by $2$. I would also like to thank Richard Hepworth for his early encouragement. The author's postdoc is funded by Gijs Heuts' ERC Starting Grant “Chromatic homotopy theory of spaces”, grant no. 950048.

\section{Algebra}

In this section we will prove the main technical result, Theorem \ref{thm:main}.

\subsection{The Mayer-Vietoris Complex} \label{subsection: MV}

Let $A$ be an $R$-algebra, $M$ be an $A$-module, and $N \subset M$ an $A$-submodule. Let $N_1, \dots , N_w$ be submodules of $N$ such that $$N_1 + \dots + N_w = N.$$

In this situation, we have a `Mayer-Vietoris' chain complex of left $A$-modules $C_*$, which we will now describe (for the classical version see \cite[Propositions 15.2 and 15.18]{BottTu}).

We will be interested in this complex in the case that $N_i=J_i$ is an idempotent cover of an ideal $N=I$ of $M=A$. It is essentially an abstract version of Sroka's \emph{cellular Davis complex} \cite[Definition 8]{Sroka}.

We set
$$C_p := \bigoplus_{\substack{S \subset \underline{w} \\ \norm{S}=p}} \bigcap_{i \in S} N_i$$ for $p = 1, \dots , w$, with $C_0 := M$, and (the augmentation) $C_{-1} := \faktor{M}{N}$. We adopt the convention that $C_p=0$ for $p>w$.

The differential $d_p:C_p \to C_{p-1}$ is defined as follows. For $p=0$ it is just the projection $M \to \faktor{M}{N}$. For $p=1$ it is the direct sum of the inclusions. For $p \geq 2$, we define $d_p$ on the summand $\bigcap_{i \in S} N_i$ to be the map $$\bigcap_{i \in S} N_i \to \bigoplus_{j \in S} \bigcap_{i \in S \setminus \{j\}} N_i,$$ $$x \mapsto \sum_{j \in S} (-1)^{\#(S,j)}\iota_{(S,j)}(x),$$ where $\#(S,j)$ is the number of elements of $S$ which are less than $j$, and $\iota_{(S,j)}$ is the inclusion $\bigcap_{i \in S} N_i \to \bigcap_{i \in S \setminus \{j\}} N_i$.

Although we presented $C_0$ and $d_1$ as special cases, they are not. One may think of $C_0$ as the intersection of \emph{none} of the $J_i$, and under this convention the definition of $d_1$ coincides with the definition of the higher differentials.

\begin{lemma} \label{lem:MVcx} The Mayer-Vietoris complex associated to a cover is a chain complex. That is, the differential satisfies $d^2=0$.
\end{lemma}

The proof is standard, and exactly parallels Sroka's proof that his cellular Davis complex is a chain complex \cite[Lemma 9]{Sroka}.

\begin{proof} Since the $N_i$ are all contained in $N$, $d_{-1} \circ d_0 = 0$.

For $p \geq 1$, it suffices to argue that the restiction of $d_{p-1} \circ d_p$ to each summand $\bigcap_{i \in S} N_i$ (with $\norm{S}=p$) of $C_p$ is zero.

Fixing such a set $S$, let $x \in \bigcap_{i \in S} N_i \subset C_p$. We compute:
\begin{align*}
    d_{p-1} \circ d_p (x) & = d_{p-1}(\sum_{j \in S} (-1)^{\#(S,j)}\iota_{(S,j)}(x)) \\
    &= \sum_{i \in S \setminus\{j\}} \sum_{j \in S} (-1)^{\#(S \setminus \{j\},i) + \#(S,j)}\iota_{(S \setminus \{j\},i)}\iota_{(S,j)}(x) \\
    &= \sum_{i \neq j \in S} (-1)^{\#(S \setminus \{j\},i) + \#(S,j)}\iota_{(S \setminus \{j\},i)}\iota_{(S,j)}(x)
\end{align*}
We have $$\iota_{(S \setminus \{j\},i)}\iota_{(S,j)} = \iota_{(S \setminus \{i\},j)}\iota_{(S,i)},$$ so it suffices to show that $$(-1)^{\#(S \setminus \{j\},i) + \#(S,j)}  = - (-1)^{\#(S \setminus \{i\},j) + \#(S,i)}.$$

This holds because $i<j$ implies that $\#(S \setminus \{i\},j) = \#(S,j) -1$, and $\#(S \setminus \{j\},i) = \#(S,i)$, and this completes the proof.
\end{proof}

\begin{lemma} \label{lem:MVac} If $N$ is free as an $R$-module on some basis $\mathcal{B}$, such that the $N_i$ are free $R$-modules on subsets of $\mathcal{B}$, then the Mayer-Vietoris complex is acyclic.
\end{lemma}

Again, the proof runs exactly parallel to Sroka's proof that his cellular Davis complex is acyclic \cite[Theorem 10]{Sroka}.

\begin{proof} By construction, $d_0:M \to \faktor{M}{N}$ is a surjection, so the Mayer-Vietoris complex is exact in degree $-1$. Exactness in degree $0$ follows from the assumption that $N_1 + \dots + N_w = N$. We may therefore restrict attention to $p \geq 1$.

Write $\mathcal{B}_i \subset \mathcal{B}$ for the basis of $N_i$. Then each intersection $\bigcap_{i \in S} N_i$ is free on the corresponding intersection $\bigcap_{i \in S} \mathcal{B}_i$.

For each $v \in \mathcal{B}$, and $S \subset \underline{w}$, let $\chi_v(S)$ be defined as follows.
\begin{itemize}
    \item If $v \in \bigcap_{i \in S} \mathcal{B}_i$, then $\chi_v(S) \cong R$ is the submodule of $\bigcap_{i \in S} N_i$ generated by $v$.
    \item If $v \not\in \bigcap_{i \in S} \mathcal{B}_i$, then $\chi_v(S) = 0$.
\end{itemize}

It is then tautological that for each $S \subset \underline{w}$ we have $$\bigcap_{i \in S} N_i = \bigoplus_{v \in \mathcal{B}} \chi_v(S),$$ so we get a decomposition of $C_p$ as an $R$-module: $$C_p = \bigoplus_{\substack{S \subset \underline{w} \\ \norm{S} = p}} \bigcap_{i \in S} N_i = \bigoplus_{\substack{S \subset \underline{w} \\ \norm{S} = p}} \bigoplus_{v \in \mathcal{B}} \chi_v(S) = \bigoplus_{v \in \mathcal{B}} \bigoplus_{\substack{S \subset \underline{w} \\ \norm{S} = p}}  \chi_v(S) = \bigoplus_{v \in \mathcal{B}} C^v_p,$$ where we define $C^v_*$ to be the chain complex with $C^v_p := \bigoplus_{\substack{S \subset \underline{w} \\ \norm{S} = p}}  \chi_v(S)$.

The boundary map $d_p$ is a linear combination of inclusions, so it automatically respects this decomposition. That is, as chain complexes in $R$-modules we have $$C_* = \bigoplus_{v \in \mathcal{B}} C^v_*,$$ in degrees $\geq 1$, and in degree zero $\bigoplus_{v \in \mathcal{B}} C^v_* = N = \mathrm{Im}(d_1) \subset M$, so it suffices to establish that each $C^v_*$ is acyclic.

Let $S(v)$ be the subset of $\underline{w}$ consisting of those $i$ such that $v \in \mathcal{B}_i$. Then, notice that $C^v_*$ is exactly the augmented chain complex (over $R$, shifted up by one degree) of the simplex $\Delta^{\norm{S(v)}-1}$. This complex is acyclic, so we are done.
\end{proof}

\begin{lemma} \label{lem:idempotentTensor} Let $X$ be a right $A$-module. Let $J$ be a left ideal of $A$ which is generated by idempotents and acts trivially on $X$. Then we have $X \otimes_A J = 0$. \end{lemma}

\begin{proof} It suffices to show that $x \otimes \alpha e = 0$ for $e$ an idempotent generator of $J$, $x \in X$ and $\alpha \in A$. We have $x \otimes \alpha e = x \otimes \alpha e^2 = x \cdot (\alpha e) \otimes e,$ which is zero since the action is trivial, as required. \end{proof}

For an $A$-module $X$, we will say that a \emph{partial projective resolution of $X$} of \emph{length $h$} is an exact sequence
$$P_h \to P_{h-1} \to \dots \to P_1 \to P_0 \to X \to 0,$$
where each $P_i$ is projective. In these terms, the results of this subsection can be summarised as follows.

\begin{proposition} \label{prop:MVisEllProjRes} Let $A$ be an $R$-algebra, let $I$ be a twosided ideal of $A$, and let $J_1, \dots ,J_w$ be an $R$-free idempotent left cover of $I$ of height $h \leq w$. The truncation $C_*^{\leq h}$ in degree $h$ of the Mayer-Vietoris complex associated to $J_1, \dots , J_w$ is a length-$h$ partial projective resolution of $\faktor{A}{I}$, with $C_0 = A$. This partial projective resolution has the additional property that $X \otimes_A C_p = 0$ for $p \geq 1$ for any right $A$-module $X$ on which $I$ acts trivially. If $h=w$ then $C_*^{\leq h}=C_*$ is a projective resolution of $\faktor{A}{I}$.
\end{proposition}

\begin{proof} Combine Lemmas \ref{lem:MVac} and \ref{lem:idempotentTensor} with the observation that a principal left ideal generated by an idempotent is a projective left $A$-module (this is standard, but see for example \cite[Lemma 3.1]{Boyde}).   
\end{proof}

In the wild, we will recognise principal idempotent ideals in the following manner.

\begin{lemma} \label{lem:retractIsIdemp} If $J$ is a left ideal of an $R$-algebra $A$ which is a retract of $A$ via a right multiplication map $A \xrightarrow{\cdot e} J$ for some element $e \in J$, then $J$ is the principal left ideal generated by $e$, and $e$ is idempotent. \end{lemma}

\begin{proof} Since $\cdot e$ is a retraction of the inclusion $J \to A$, we have $x \cdot e = x$ for all $x \in J$. This gives that $J$ is principal and generated by $e$. In particular, since $e$ is itself in $J$, we have $e \cdot e = e$, that is, $e$ is idempotent.
\end{proof}

\subsection{The stable isomorphism} \label{subs:StableIso}

\begin{proposition} \label{prop:SpecSeq} Let $M$ be a left $A$-module, let $N$ be a right $A$-module, and let $h>0$. Suppose that there exists a length-$h$ partial projective resolution $C_*^{\leq h}$ of $M$ with the property that $N \otimes_A C_p^{\leq h} = 0$ for $p \geq 1$. Then $$\Tor_q^{A}(N,M)=0$$ for $0< q \leq h$, and $N \otimes_A M \cong N \otimes_A C_0^{\leq h}$.
\end{proposition}

\begin{proof} Let $C_*$ be an extension of $C_*^{\leq h}$ to a projective resolution of $M$ and form the tensor product $N \otimes_A C_*$. The homology in degree $q$, $H_q(N \otimes_A C_*)$ is by definition $\Tor_*^{A}(N,M)$, which vanishes in the claimed range since $N \otimes_A C_*$ does. Since $h > 0$, we have $N \otimes_A C_1 = N \otimes_A C_1^{\leq h} = 0$, so $H_0(N \otimes_A C_*) \cong N \otimes_A C_0^{\leq h}$. Since $H_0(N \otimes_A C_*) = \Tor_0^A(N,M) = N \otimes_A M$, the result follows.
\end{proof}

\begin{theorem} \label{thm:MN} Let $N$ be a left $A$-module, let $M$ be a right $A$-module, and let $h>0$. Let $I$ be a twosided ideal of $A$ which acts trivially on $M$ and $N$. Suppose that there exists a length-$h$ partial projective resolution $C_*^{\leq h}$ of $\faktor{A}{I}$, with $C_0^{\leq h} = A$ and $N \otimes_A C_p^{\leq h} = 0$ for $p \geq 1$. Then the natural map $$\Tor_q^A(N,M) \to \Tor_q^{\faktor{A}{I}}(N,M)$$ is an isomorphism for $q \leq h$ and a surjection for $q=h+1$.
\end{theorem}

Note that the hypothesis that $I$ acts trivially on $M$ and $N$ means that $M$ and $N$ are $\faktor{A}{I}$-modules, and hence that the expression $\Tor_q^{\faktor{A}{I}}(N,M)$ is meaningful.

\begin{proof} We will argue that $\Tor_q^A(N,M)$ and $\Tor_q^{\faktor{A}{I}}(N,M)$ are the homology of the same chain complex.

Let $P_*$ be a free $A$-resolution of $N$. Then $\Tor_*^A(N,M)$ is the homology of $P_* \otimes_A M$. Since $I$ acts trivially on $M$, we may write $$P_* \otimes_A M \cong P_* \otimes_A (\faktor{A}{I} \otimes_{\faktor{A}{I}} M).$$

Then, since $P_*$ is free over $A$, the tensor product $P_* \otimes_A \faktor{A}{I}$ is free over $\faktor{A}{I}$. By Proposition \ref{prop:SpecSeq}, we have $\Tor_q^A(N,\faktor{A}{I})=0$ for $0<q \leq h$, which is to say that the homology of $P_* \otimes_A \faktor{A}{I}$ vanishes in degrees $0<q \leq h$.

By the same proposition, $$\Tor_0^A(N,\faktor{A}{I}) := N \otimes_A \faktor{A}{I} \cong N \otimes_A C_0 \cong N,$$ since $C_0 = A$ by assumption.

Thus, $P_*^{\leq h+1} \otimes_A \faktor{A}{I}$ is a length-$(h+1)$ partial projective resolution of $N$ over $\faktor{A}{I}$, so $P_* \otimes_A \faktor{A}{I} \otimes_{\faktor{A}{I}} M$ computes $\Tor_q^{\faktor{A}{I}}(N,M)$ for $q \leq h$, so we get the desired isomorphism for $q \leq h$.

To get the surjection in degree $h+1$, first add a free $\faktor{A}{I}$-module $X$ to $P_* \otimes_A \faktor{A}{I}$ in degree $h+2$, in such a way that the resulting complex $[P_* \otimes_A \faktor{A}{I}] \oplus X$ is exact in degree $h+1$ (i.e. to kill the kernel of the differential in degree $h+1$). By construction, taking homology of the complex $([P_* \otimes_A \faktor{A}{I}] \oplus X) \otimes_{\faktor{A}{I}} M$ computes $\Tor_q^{\faktor{A}{I}}(N,M)$ for $q \leq h+1$. The cokernel of the injection
$$P_* \otimes_A M \cong P_* \otimes_A \faktor{A}{I} \otimes_{\faktor{A}{I}} M \to ([P_* \otimes_A \faktor{A}{I}] \oplus X) \otimes_{\faktor{A}{I}} M$$
is just a copy of $X \otimes_{\faktor{A}{I}} M$ in degree $h+2$. In particular, this cokernel is zero in degrees $q \leq h+1$, so the surjectivity follows from the long exact sequence on homology associated to this short exact sequence of chain complexes.
\end{proof}

We are now ready to prove our main algebraic result.

\begin{proof}[Proof of Theorem \ref{thm:main}] Combine Theorem \ref{thm:MN} and Proposition \ref{prop:MVisEllProjRes}. \end{proof}

\section{Partition algebras}

We will think of the partition algebras $\Partition_n(\delta)$ in more or less the same way as Boyd-Hepworth-Patzt \cite[Definition 2.1]{BHPPartition}. We will use slightly different notation for our vertex set, inspired by Xi's paper \cite{Xi}. Xi proves that the partition algebras have a \emph{cellular} structure, and our approach is sometimes motivated by this, but we will not need to make it explicit.

We write $\mathbb{N}_0$ for the set of non-negative integers, and $\mathbb{N}$ for the set of positive integers.

\begin{definition} \label{def:Priming} Consider the set $\mathbb{N} \times \mathbb{N}_0$. We will write $(x,y) \in \mathbb{N} \times \mathbb{N}_0$ as 
$$x'{^{\dots}}'$$
i.e. as $x$ with a $y$-fold prime symbol. For example, $(3,2)$ is written $3''$, $(1,0)$ is written $1$, and $(5,4)$ is written $5''''$. For a subset $S$ of $\mathbb{N} \times \mathbb{N}_0$, let $S'$ be the set $\{(x,y+1) | (x,y) \in S\}$; in our notation, this is the set $\{x' | x \in S\}$.

For a fixed $n$, we write $\underline{n}$ for the subset $\{1,2, \dots, n\}$ of $\mathbb{N} \times \mathbb{N}_0$, so that $\underline{n}' = \{1',2', \dots, n'\}$ is in bijection with $\underline{n}$ via the map which `primes each element'.
\end{definition}

With this notation in hand, we define the partition algebras as follows.

\begin{definition} \label{def:partitionAlg} Let $R$ be a commutative ring with unit, let $\delta \in R$, and let $n \geq 1$. The \emph{partition algebra} $\Partition_n(\delta)$ is the free $R$-module on the set of partitions $\rho$ of $$\underline{n} \cup \underline{n}' = \{1,2,\dots,n,1',2', \dots, n'\}.$$
The multiplication is defined on partitions, then extended bilinearly, as follows. For partitions $\mu$ and $\nu$ of $\underline{n} \cup \underline{n}'$, let $\nu'$ be the partition of $(\underline{n} \cup \underline{n}')' = \underline{n}' \cup \underline{n}''$ corresponding to $\nu$ under the priming-unpriming correspondence of Definition \ref{def:Priming}. Let the \emph{composed partition}
$$\mu \circ \nu'$$
be the finest partition of $\underline{n} \cup \underline{n}' \cup \underline{n}''$ whose restriction to $\underline{n} \cup \underline{n}'$ is coarser than $\mu$ and whose restriction to $\underline{n}' \cup \underline{n}''$ is coarser than $\nu'$. Let $j \in \mathbb{N}_0$ be the number of parts of $\mu \circ \nu'$ containing only elements of $\underline{n}'$, and let $\mu * \nu$ be the partition of $\underline{n} \cup \underline{n}'$ obtained by restricting $\mu \circ \nu'$ to $\underline{n} \cup \underline{n}''$ and then `unpriming' $\underline{n}''$. The product $\mu \nu$ in the partition algebra is then defined to be $$\mu \nu = \delta^j \cdot \mu * \nu \in \Partition_n(\delta).$$
We will feel free to call $\mu * \nu$ the \emph{underlying partition} of the product $\mu \nu$, since it is well defined even if $\delta^j=0$.
\end{definition}

In this section we will construct an idempotent left cover of height $n-1$ for the the ideal $I_{\leq n-1}$ of the partition algebra $\Partition_n(\delta)$. Recall that $I_{\leq n-1}$ is the $R$-span of partitions having at most $n-1$ (i.e. fewer than the maximal number $n$) parts containing both primed and unprimed elements.

\begin{remark} \label{rmk: Partition trivial module} The partition algebra $\Partition_n(\delta)$ comes with a natural trivial module (equivalently, an augmentation), which we denote $\t$. This is a single copy of $R$ where partitions with (at least) $n$ parts containing both a primed and an unprimed element act by 1, and other partitions (i.e. those in $I_{\leq n-1}$) act by zero. Give the symmetric group algebra $R \Sigma_n$ its usual trivial module, which we also denote $\t$. Note that the quotient map $\Partition_n(\delta) \to R \Sigma_n$ given in the introduction respects these actions (as does the inclusion $R \Sigma_n \to \Partition_n(\delta)$), and since we have by construction that $I_{\leq n-1}$ acts trivially on $\t$, this quotient map induces an isomorphism $\faktor{\Partition_n(\delta)}{I_{\leq n-1}} \cong R \Sigma_n$ of augmented algebras.   
\end{remark}

\begin{definition} For $i$ in $\underline{n}$ let $K_i$ be the left ideal of $\Partition_n(\delta)$ spanned by partitions where the vertex $i'$ on the right is a singleton. For distinct elements $i<j$ in $\underline{n}$, let $L_{i,j}$ be the left ideal spanned by diagrams where the vertices $i'$ and $j'$ on the right belong to the same part of the partition. \end{definition}

\begin{lemma} \label{lem:cover} The ideals $K_i$ and $L_{i,j}$ cover $I_{\leq n-1}$.
\end{lemma}

\begin{proof} It suffices to argue that
\begin{enumerate}
    \item each $K_i$ or $L_{i,j}$ is contained in $I_{\leq n-1}$, and
    \item any partition with fewer than $n$ parts containing both primed and unprimed elements belongs to some $K_i$ or to some $L_{i,j}$.
\end{enumerate}

For the first point, simply note that a partition belonging to $K_i$ or $L_{i,j}$ has at most $n-1$ parts that contain both primed and unprimed elements, so the partition itself belongs to $I_{\leq n-1}$. For the second point, suppose that a partition $\rho$ has fewer than $n$ parts with both primed and unprimed elements. If two primed elements belong to the same part of $\rho$ then $\rho$ lies in some $L_{i,j}$. If no two primed elements belong to the same part, then $\rho$ must have at least $n$ parts, and by assumption at least one of these parts cannot contain an unprimed element. It follows that the part in question is a singleton, and hence that $\rho$ lies in some $K_i$. This completes the proof. 
\end{proof}

Unsurprisingly, we will be concerned with iterated intersections of these ideals. The following lemma is immediate. Write $\underline{n}^2_{<}$ for the subset of $\underline{n} \times \underline{n}$ consisting of pairs $(i,j)$ with $i<j$ (i.e. the indexing set for the ideals $L_{i,j}$).

\begin{lemma} \label{lem:theFormOfTheIntersection} Let $S \subset\underline{n}$, and let $T \subset \underline{n}^2_{<}$. The intersection $$\bigcap_{i \in S} K_i \cap \bigcap_{(i,j) \in T} L_{i,j}$$ is the $R$-span of those partitions for which
\begin{itemize}
    \item $i'$ is a singleton whenever $i \in S$, and
    \item $i'$ and $j'$ lie in the same part of the partition whenever $(i,j) \in T$
\end{itemize}
In particular, the intersection is zero if and only if there exists $(i,j) \in T$ so that $i \in S$ or $j \in S$. \qed
\end{lemma}

\begin{lemma} \label{lem:PmapK} Let $S \subset\underline{n}$, and let $T \subset \underline{n}^2_{<}$. Let $$J = \bigcap_{i \in S} K_i \cap \bigcap_{(i,j) \in T} L_{i,j},$$ and let $a \in \underline{n} \setminus S$ be such that $S \cup \{a\}$ is a proper subset of $\underline{n}$. Choose $b \in \underline{n} \setminus (S \cup \{a\})$, and let $\mu$ be the partition of $\underline{n} \cup \underline{n}'$ whose parts are
\begin{itemize}
    \item the singleton $\{a'\}$,
    \item the triple $\{a, b, b'\}$
    \item the pair $\{i,i'\}$ for each $i \in \underline{n} \setminus \{a, b\}$.
\end{itemize}
Then either $K_{a} \cap J = 0$ or $J \cdot \mu \subset  K_{a} \cap J$.
\end{lemma}

Of course, $\mu$ depends on $a$ and $b$, but we take the liberty of omitting this from the notation.

To give the idea, here is a picture. We follow the convention of \cite{BHPPartition} and represent partitions as graphs - vertices connected by a sequence of edges lie in the same part of the partition. Set $n = 7$, $S = \{6\}$, $T=\{(1,2),(3,4)\}$, $a=5$. Choose $b=4$, so that $\mu$ is the below partition on the right, and suppose given $\rho$, the (reasonably generic) partition on the left:

\begin{center}
$\rho =$
\quad
\begin{tikzpicture}[x=1.3cm,baseline=1.5cm]

\def\wid{2}
\def\cheatWid{\wid*1.05}
\def\hei{0.5}
\def\nodesize{3}

\node[v, minimum size=\nodesize] (a1) at (1* \wid,0*\hei) {};
\node[v, minimum size=\nodesize] (a2) at (1* \wid,1*\hei) {};
\node[v, minimum size=\nodesize] (a3) at (1* \wid,2*\hei) {};
\node[v, minimum size=\nodesize] (a4) at (1* \wid,3*\hei) {};
\node[v, minimum size=\nodesize] (a5) at (1* \wid,4*\hei) {};
\node[v, minimum size=\nodesize] (a6) at (1* \wid,5*\hei) {};
\node[v, minimum size=\nodesize] (a7) at (1* \wid,6*\hei) {};

\node[v, minimum size=\nodesize] (x1) at (1.25* \wid,1*\hei) {};

\node[v, minimum size=\nodesize] (y1) at (1.75* \wid,0.5*\hei) {};
\node[v, minimum size=\nodesize] (y2) at (1.75* \wid,2.5*\hei) {};

\node[v, minimum size=\nodesize] (b1) at (2* \wid,0*\hei) {};
\node[v, minimum size=\nodesize] (b2) at (2* \wid,1*\hei) {};
\node[v, minimum size=\nodesize] (b3) at (2* \wid,2*\hei) {};
\node[v, minimum size=\nodesize] (b4) at (2* \wid,3*\hei) {};
\node[v, minimum size=\nodesize] (b5) at (2* \wid,4*\hei) {};
\node[v, minimum size=\nodesize] (b6) at (2* \wid,5*\hei) {};
\node[v, minimum size=\nodesize] (b7) at (2* \wid,6*\hei) {};

\draw[e] (a7) to (b5);
\draw[e] (a6) to (b7);

\draw[e] (a4) to (x1);
\draw[e] (a3) to (x1);
\draw[e] (a1) to (x1);

\draw[e] (x1) to (y1);

\draw[e] (y1) to (b1);
\draw[e] (y1) to (b2);
\draw[e] (y2) to (b3);
\draw[e] (y2) to (b4);

\draw (2.4* \wid,5*\hei) node{$S$};
\draw[e, dotted] (b6) to (2.3* \wid,5*\hei);

\draw (2.4* \wid,1.5*\hei) node{$T$};
\draw[e, dotted] (2.1* \wid,2.5*\hei) to (2.3* \wid,1.5*\hei);

\draw (2.4* \wid,1.5*\hei) node{$T$};
\draw[e, dotted] (2.1* \wid,0.5*\hei) to (2.3* \wid,1.5*\hei);

\draw (0.8* \wid,6*\hei) node{$7$};
\draw (0.8* \wid,5*\hei) node{$6$};
\draw (0.8* \wid,4*\hei) node{$5$};
\draw (0.8* \wid,3*\hei) node{$4$};
\draw (0.8* \wid,2*\hei) node{$3$};
\draw (0.8* \wid,1*\hei) node{$2$};
\draw (0.8* \wid,0*\hei) node{$1$};

\end{tikzpicture}
,
\quad
$\mu =$
\quad
\begin{tikzpicture}[x=1.5cm,baseline=1.5cm]

\def\wid{2}
\def\cheatWid{\wid*1.05}
\def\hei{0.5}
\def\nodesize{3}

\node[v, minimum size=\nodesize] (b1) at (2* \wid,0*\hei) {};
\node[v, minimum size=\nodesize] (b2) at (2* \wid,1*\hei) {};
\node[v, minimum size=\nodesize] (b3) at (2* \wid,2*\hei) {};
\node[v, minimum size=\nodesize] (b4) at (2* \wid,3*\hei) {};
\node[v, minimum size=\nodesize] (b5) at (2* \wid,4*\hei) {};
\node[v, minimum size=\nodesize] (b6) at (2* \wid,5*\hei) {};
\node[v, minimum size=\nodesize] (b7) at (2* \wid,6*\hei) {};

\node[v, minimum size=\nodesize] (c1) at (3* \wid,0*\hei) {};
\node[v, minimum size=\nodesize] (c2) at (3* \wid,1*\hei) {};
\node[v, minimum size=\nodesize] (c3) at (3* \wid,2*\hei) {};
\node[v, minimum size=\nodesize] (c4) at (3* \wid,3*\hei) {};
\node[v, minimum size=\nodesize] (c5) at (3* \wid,4*\hei) {};
\node[v, minimum size=\nodesize] (c6) at (3* \wid,5*\hei) {};
\node[v, minimum size=\nodesize] (c7) at (3* \wid,6*\hei) {};

\node[v, minimum size=\nodesize] (z1) at (2.5* \wid,3*\hei) {};

\draw[e] (b1) to (c1);
\draw[e] (b2) to (c2);
\draw[e] (b3) to (c3);
\draw[e] (b6) to (c6);
\draw[e] (b7) to (c7);

\draw[e] (b4) to (z1);
\draw[e] (b5) to (z1);

\draw[e] (z1) to (c4);

\draw (3.4* \wid,4*\hei) node{$a'$};
\draw[e, dotted] (c5) to (3.3* \wid,4*\hei);

\draw (3.4* \wid,3*\hei) node{$b'$};
\draw[e, dotted] (c4) to (3.3* \wid,3*\hei);

\draw (1.8* \wid,6*\hei) node{$7$};
\draw (1.8* \wid,5*\hei) node{$6$};
\draw (1.8* \wid,4*\hei) node{$5$};
\draw (1.8* \wid,3*\hei) node{$4$};
\draw (1.8* \wid,2*\hei) node{$3$};
\draw (1.8* \wid,1*\hei) node{$2$};
\draw (1.8* \wid,0*\hei) node{$1$};
\end{tikzpicture}
\end{center}

The product $\rho \mu$ then looks as follows, where we highlight in red the `mechanisms' by which the features of $\rho$ indexed by $S$ and $T$ are preserved under multiplication by $\mu$:
\begin{center}
$\rho \mu =$
\quad
\begin{tikzpicture}[x=1.3cm,baseline=1.5cm]

\def\wid{2}
\def\cheatWid{\wid*1.05}
\def\hei{0.5}
\def\nodesize{3}

\node[v, minimum size=\nodesize] (a1) at (1* \wid,0*\hei) {};
\node[v, minimum size=\nodesize] (a2) at (1* \wid,1*\hei) {};
\node[v, minimum size=\nodesize] (a3) at (1* \wid,2*\hei) {};
\node[v, minimum size=\nodesize] (a4) at (1* \wid,3*\hei) {};
\node[v, minimum size=\nodesize] (a5) at (1* \wid,4*\hei) {};
\node[v, minimum size=\nodesize] (a6) at (1* \wid,5*\hei) {};
\node[v, minimum size=\nodesize] (a7) at (1* \wid,6*\hei) {};

\node[v, minimum size=\nodesize] (x1) at (1.25* \wid,1*\hei) {};

\node[v, minimum size=\nodesize] (y1) at (1.75* \wid,0.5*\hei) {};
\node[v, minimum size=\nodesize] (y2) at (1.75* \wid,2.5*\hei) {};

\node[v, minimum size=\nodesize] (b1) at (2* \wid,0*\hei) {};
\node[v, minimum size=\nodesize] (b2) at (2* \wid,1*\hei) {};
\node[v, minimum size=\nodesize] (b3) at (2* \wid,2*\hei) {};
\node[v, minimum size=\nodesize] (b4) at (2* \wid,3*\hei) {};
\node[v, minimum size=\nodesize] (b5) at (2* \wid,4*\hei) {};
\node[v, minimum size=\nodesize] (b6) at (2* \wid,5*\hei) {};
\node[v, minimum size=\nodesize] (b7) at (2* \wid,6*\hei) {};

\draw[e] (a7) to (b5);
\draw[e] (a6) to (b7);

\draw[e] (a4) to (x1);
\draw[e] (a3) to (x1);
\draw[e] (a1) to (x1);

\draw[e] (x1) to (y1);

\draw[e, red] (y1) to (b1);
\draw[e, red] (y1) to (b2);
\draw[e, red] (y2) to (b3);
\draw[e, red] (y2) to (b4);

\def\wid{2}
\def\cheatWid{\wid*1.05}
\def\hei{0.5}
\def\nodesize{3}

\node[v, minimum size=\nodesize] (c1) at (3* \wid,0*\hei) {};
\node[v, minimum size=\nodesize] (c2) at (3* \wid,1*\hei) {};
\node[v, minimum size=\nodesize] (c3) at (3* \wid,2*\hei) {};
\node[v, minimum size=\nodesize] (c4) at (3* \wid,3*\hei) {};
\node[v, minimum size=\nodesize] (c5) at (3* \wid,4*\hei) {};
\node[v, minimum size=\nodesize] (c6) at (3* \wid,5*\hei) {};
\node[v, minimum size=\nodesize] (c7) at (3* \wid,6*\hei) {};

\node[v, minimum size=\nodesize] (z1) at (2.5* \wid,3*\hei) {};

\draw[e, red] (b1) to (c1);
\draw[e, red] (b2) to (c2);
\draw[e, red] (b3) to (c3);
\draw[e, red] (b6) to (c6);
\draw[e] (b7) to (c7);

\draw[e, red] (b4) to (z1);
\draw[e] (b5) to (z1);

\draw[e,red] (z1) to (c4);
\end{tikzpicture}
\quad
$=$
\quad
\begin{tikzpicture}[x=1.3cm,baseline=1.5cm]

\def\wid{2}
\def\cheatWid{\wid*1.05}
\def\hei{0.5}
\def\nodesize{3}

\node[v, minimum size=\nodesize] (a1) at (1* \wid,0*\hei) {};
\node[v, minimum size=\nodesize] (a2) at (1* \wid,1*\hei) {};
\node[v, minimum size=\nodesize] (a3) at (1* \wid,2*\hei) {};
\node[v, minimum size=\nodesize] (a4) at (1* \wid,3*\hei) {};
\node[v, minimum size=\nodesize] (a5) at (1* \wid,4*\hei) {};
\node[v, minimum size=\nodesize] (a6) at (1* \wid,5*\hei) {};
\node[v, minimum size=\nodesize] (a7) at (1* \wid,6*\hei) {};

\node[v, minimum size=\nodesize] (x1) at (1.25* \wid,1*\hei) {};

\node[v, minimum size=\nodesize] (y2) at (1.75* \wid,0.5*\hei) {};

\node[v, minimum size=\nodesize] (c1) at (2* \wid,0*\hei) {};
\node[v, minimum size=\nodesize] (c2) at (2* \wid,1*\hei) {};
\node[v, minimum size=\nodesize] (c3) at (2* \wid,2*\hei) {};
\node[v, minimum size=\nodesize] (c4) at (2* \wid,3*\hei) {};
\node[v, minimum size=\nodesize] (c5) at (2* \wid,4*\hei) {};
\node[v, minimum size=\nodesize] (c6) at (2* \wid,5*\hei) {};
\node[v, minimum size=\nodesize] (c7) at (2* \wid,6*\hei) {};

\node[v, minimum size=\nodesize] (z1) at (1.75* \wid,2.5*\hei) {};

\draw[e] (a7) to (z1);
\draw[e] (a6) to (c7);

\draw[e] (a4) to (x1);
\draw[e] (a3) to (x1);
\draw[e] (a1) to (x1);

\draw[e] (x1) to (y1);

\draw[e] (y1) to (c1);
\draw[e] (y1) to (c2);
\draw[e] (z1) to (c3);
\draw[e] (z1) to (c4);

\draw (2.4* \wid,4*\hei) node{$a'$};
\draw[e, dotted] (c5) to (2.3* \wid,4*\hei);

\draw (2.4* \wid,3*\hei) node{$b'$};
\draw[e, dotted] (c4) to (2.3* \wid,3*\hei);

\draw (2.4* \wid,5*\hei) node{$S$};
\draw[e, dotted] (c6) to (2.3* \wid,5*\hei);

\draw (2.4* \wid,1.5*\hei) node{$T$};
\draw[e, dotted] (2.1* \wid,2.5*\hei) to (2.3* \wid,1.5*\hei);

\draw (2.4* \wid,1.5*\hei) node{$T$};
\draw[e, dotted] (2.1* \wid,0.5*\hei) to (2.3* \wid,1.5*\hei);
\end{tikzpicture}
\end{center}

Formally, the proof goes as follows.

\begin{proof} Assume $K_{a} \cap J \neq 0$, and let $\rho \in J$ be a partition. We must show that $\rho \mu \in K_{a} \cap J$.

If $i \in S$ then, since $\rho \in J$, $i'$ is a singleton in $\rho$. By construction, $i$ is not equal to $a$ or $b$, so $i'$ is still a singleton in (the underlying partition of) $\rho \mu$. It follows that $\rho \mu$ is in $K_i$.

If $(i,j) \in T$ then, since $\rho \in J$, $i'$ and $j'$ lie in the same part of $\rho$. If either $i = a$ or $j = a$ then by Lemma \ref{lem:theFormOfTheIntersection} we have $K_a \cap J = 0$, which contradicts our initial assumption, so in fact neither $i$ nor $j$ can be equal to $a$. It follows that $i'$ and $j'$ lie in the same part of $\rho \mu$, hence that $\rho \mu$ is in $L_{i,j}$.

Lastly, $a'$ is a singleton in $\mu$, hence also in $\rho \mu$, so $\rho \mu \in K_a$. In total, we have shown that $\rho \mu \in \bigcap_{i \in S \cup \{a \}} K_i \cap \bigcap_{(i,j) \in T} L_{i,j} = K_{a} \cap J$, as required.
\end{proof}

\begin{lemma} \label{lem:PretractK} Let $S \subset\underline{n}$, and let $T \subset \underline{n}^2_{<}$. Let $$J = \bigcap_{i \in S} K_i \cap \bigcap_{(i,j) \in T} L_{i,j},$$ and let $a \in \underline{n} \setminus S$ be such that $S \cup \{a\}$ is a proper subset of $\underline{n}$. 
If $K_a \cap J$ is nonzero, then right multiplication by the element $\mu$ constructed in Lemma \ref{lem:PmapK} gives a retraction of the inclusion of $K_a \cap J$ into $J.$
\end{lemma}

\begin{proof} We must take a partition $\rho \in K_a \cap J$ and show that $\rho \mu = \rho$. 

Every unprimed vertex lies in the same part of $\mu$ as some primed vertex, so for any partition $\rho$ the product $\rho \mu$ produces no factors of $\delta$, and is again a partition. We must argue that this partition is $\rho$.

Let $A$ be the part of $\rho$ containing $a'$, and let $B$ be the part of $\rho$ containing $b'$. It then follows from the definition of $\mu$ that $\rho \mu$ has parts:
\begin{itemize}
    \item the singleton $\{a'\}$
    \item $(A \setminus \{a'\}) \cup B$
    \item the parts of $\rho$ other than $A$ and $B$.
\end{itemize}

In other words, $\mu$ operates on $\rho$ by removing $a'$ from its part and merging the remainder with the part containing $b'$.

It follows that if $a'$ is already a singleton in $\rho$, then $\rho \mu = \rho$. \end{proof}

\begin{lemma} \label{lem:PMapL} Let $S \subset\underline{n}$, and let $T \subset \underline{n}^2_{<}$. Let $$J = \bigcap_{i \in S} K_i \cap \bigcap_{(i,j) \in T} L_{i,j},$$ and let $(a,b) \in \underline{n}^2_< \setminus T$. Let $\nu$ be the partition of $\underline{n} \cup \underline{n}'$ whose parts are
\begin{itemize}
    \item the quadruple $\{a, b, a', b'\}$, and
    \item the pair $\{i,i'\}$ for each $i \in \underline{n} \setminus \{a, b\}$.
\end{itemize}
Then either $L_{a,b} \cap J = 0$ or $J \cdot \nu \subset L_{a,b} \cap J$. \end{lemma}

\begin{proof} Assume $L_{a,b} \cap J \neq 0$, and let $\rho \in J$ be a partition. We must show that $\rho \nu \in L_{a,b} \cap J$.

If $i \in S$ then $i'$ is a singleton in $\rho$. By Lemma \ref{lem:theFormOfTheIntersection}, the assumption $L_{a,b} \cap J \neq 0$ gives that $a$ and $b$ are not in $S$, so $i$ is not equal to $a$ or $b$, so $\{i,i'\}$ is a part of $\nu$, and $i'$ is again a singleton in $\rho \nu$. This gives that $\rho \nu \in K_i$, as required.

If $(i,j) \in T$ then $i'$ and $j'$ lie in the same part of $\rho$. For each $k \in \underline{n}$, $k$ and $k'$ lie in the same part of $\nu$, so $i'$ and $j'$ lie in the same part of $\rho \nu$, so $\rho \nu \in L_{i,j}$, as required.

Lastly, $a',b'$ lie in the same part of $\nu$, hence also in the same part of $\rho \nu$, so $\rho \nu \in L_{a,b}$. In total, we have shown that $\rho \nu \in L_{a,b} \cap J$, as required.
\end{proof}

\begin{lemma} \label{lem:PretractL} Let $S \subset\underline{n}$, and let $T \subset \underline{n}^2_{<}$. Let $$J = \bigcap_{i \in S} K_i \cap \bigcap_{(i,j) \in T} L_{i,j},$$ and let $(a,b) \in \underline{n}^2_< \setminus T$. If $L_{a,b} \cap J$ is nonzero, then right multiplication by the element $\nu$ constructed in Lemma \ref{lem:PMapL} gives a retraction of the inclusion of $L_{a,b} \cap J$ into $J$. \end{lemma}

\begin{proof} Lemma \ref{lem:PMapL} shows that right multiplication by $\nu$ has the correct codomain, so it suffices to take a partition $\rho \in L_{a,b}$ and show that $\rho \nu = \rho$.

Again, every unprimed vertex lies in the same part of $\nu$ as some primed vertex, so the product $\rho \nu$ produces no factors of $\delta$, and is again a partition. We must argue that this partition is $\rho$.

For any $\rho$, the product $\rho \nu$ is the partition obtained from $\rho$ by merging the part containing $a'$ with the part containing $b'$. In particular, if $a'$ and $b'$ are already in the same part of $\rho$ (i.e. if $\rho \in L_{a,b}$) then $\rho \nu = \rho$. \end{proof}

\begin{lemma} \label{lem:IdempotentUnlessSisn} If $S \neq \underline{n}$, then the left ideal $$J=\bigcap_{i \in S} K_i \cap \bigcap_{(i,j) \in T} L_{i,j}$$ is (zero or) principal and generated by an idempotent. \end{lemma}

\begin{proof} If $J$ is zero, the statement is trivial. If $J$ is non-zero, then $S$ and $\bigcup_{(i,j) \in T} \{i,j\}$ are disjoint by Lemma \ref{lem:theFormOfTheIntersection}. We can inductively apply Lemma \ref{lem:PretractK} to get a retraction $\Partition_n(\delta) \to \bigcap_{i \in S} K_i$, then Lemma \ref{lem:PretractL} to get a retraction $\bigcap_{i \in S} K_i \to J$. The assumption $S \neq \underline{n}$ guarantees that we can apply Lemma \ref{lem:PretractK} without violating the proper subset condition. The composite of these two retractions gives a retraction $\Partition_n(\delta) \to J$.

By construction, this retraction is a composite of maps given by right multiplication by certain elements of $\Partition_n(\delta)$, so is itself given by right multiplication by the product of those elements. This product element $\rho$ must lie in $J$ because $\rho = 1 \cdot \rho$ lies in $\textrm{Im}(\cdot \rho) \subset J$. The result then follows from Lemma \ref{lem:retractIsIdemp}.
\end{proof}

We are now ready to prove the main result on $\Partition_n(\delta)$. Recall from the introduction that we have $$\faktor{\Partition_n(\delta)}{I_{\leq n-1}} \cong R \Sigma_n,$$ the group algebra of the symmetric group, and recall from Remark \ref{rmk: Partition trivial module} that this isomorphism identifies the trivial modules (both denoted $\t$) of these two algebras, in particular that $I_{\leq n-1}$ acts trivially on the trivial module $\t$ of $\Partition_n(\delta)$.

\begin{proof}[Proof of Theorem \ref{thm:partition}] By Lemma \ref{lem:cover}, these ideals do indeed form an $R$-free cover of $I_{\leq n-1}$. By Lemma \ref{lem:IdempotentUnlessSisn}, an intersection of at most $n-1$ ideals from among the $K_i$ and $L_{i,j}$ is either zero or principal idempotent, so this is indeed a principal idempotent cover of height $n-1$. Certainly $I_{\leq n-1}$ acts trivially on $\t$, so the result follows by Theorem \ref{thm:main}. \end{proof}

\section{Jones annular algebras}

Again, we will use the `priming' convention of Definition \ref{def:Priming} for vertex labelling.

\begin{definition} \label{def:Jones} Consider a cylinder $C = S^1 \times [0,1]$. Embed the unprimed vertices $\underline{n}$, equally spaced, around $S^1 \times \{0\} \subset C$, and embed the primed vertices $\underline{n}'$ around $S^1 \times \{1\} \subset C$. Precisely, regarding $S^1$ as the complex unit circle, we embed $j$ at $(e^{i\frac{2 \pi j}{n}},0)$, and $j'$ at $(e^{i\frac{2 \pi j}{n}},1)$. This naturally identifies the unprimed vertices with a copy of the cyclic group $C_n$, and likewise for the primed vertices.

Let $R$ be a ring (as always, commutative and unital) and let $\delta \in R$. Recall (from e.g. \cite{BHP}) that the Brauer algebra $\Brauer_n(\delta)$ is the subalgebra of the partition algebra $\Partition_n(\delta)$ which has a basis consisting of partitions of $\underline{n} \cup \underline{n}'$ such that each part has cardinality 2. 

Let $\rho$ be such a partition. Say that a \emph{graphical representative of $\rho$ on the annulus} is a choice, for each part of $\rho$, of an embedded curve in $C$ connecting the two vertices of the part. Say that $\rho$ \emph{admits an annular representative} or that $\rho$ \emph{is an annular diagram} if there exists a graphical representative of $\rho$ on the annulus for which no two of the embedded curves intersect. Recall that two such representatives need not be isotopic, but their isotopy classes will differ by a Dehn twist.

The Jones annular algebra $\Jones_n(\delta)$ is then the subalgebra of $\Brauer_n(\delta)$ spanned by partitions which admit annular representatives. \end{definition}

For the sake of fluidity we will feel at liberty to think of the vertices as being labelled either by the cyclic group $C_n$ or by $\underline{n}$.

By the \emph{cyclic interval} $[a,b]$ in the cyclic group $C_n$ we mean the set $\{a,a+1,a+2, \dots, b \}$. Open cyclic intervals are defined similarly. Graham and Lehrer \cite[Proposition 6.14]{GrahamLehrer} give a useful description of the canonical basis of $\Jones_n(\delta)$, which we will now describe.

\begin{definition} \label{def:AnnularLinkState} Let $p$ be a partition of $C_n \cong \underline{n}$ into parts of cardinality 1 and 2. Call parts of cardinality 1 the \emph{defect parts}. Then $p$ is called an \emph{annular link state} if for all non-defect parts $\{i,j\}$ of $p$ we have:
    \begin{itemize}
        \item No part of $p$ having cardinality 2 consists of one element from the cyclic interval $(i,j)$ and one from $(j,i)$. In other words, $(i,j)$ and $(j,i)$ are unions of parts of $p$.
        \item Either all defect parts of $p$ (equivalently, all singletons) are contained in $(i,j)$, or all defect parts of $p$ are contained in $(j,i)$.
    \end{itemize}
The $t$ defect vertices may be ordered via the correspondence $C_n \cong \underline{n}$, and we may speak of the $i$-th defect vertex for $i \in \underline{t}$. Write $M(t)$ for the set of annular link states with $t$ defects.
\end{definition}

Note that the first condition says precisely that $p$ is \emph{non-crossing} as a partition \cite{Kreweras}. The next proposition is essentially \cite[Proposition 6.14]{GrahamLehrer}.

\begin{proposition} \label{prop:JCellular} For $0 \leq t \leq n$, $\sigma \in C_t$, and annular link states $p$ and $q$ having $t$ defects, there is a unique annular diagram $C_{p,q}^{\sigma}$ on $\underline{n} \cup \underline{n}' = C_n \cup C_n'$ such that:
\begin{itemize}
    \item the restriction of $C_{p,q}^{\sigma}$ to $\underline{n} \cong C_n$ is $p$,
    \item the restriction of $C_{p,q}^{\sigma}$ to $\underline{n}' \cong C_n'$ is identified with $q$ under the priming-unpriming correspondence, and
    \item the $i$-th defect vertex of $q$ and the $\sigma(i)$-th vertex of $p$ are connected by an edge in $C_{p,q}^{\sigma}$.
\end{itemize}
Furthermore, the resulting map $$\coprod_{0 \leq t \leq n} M(t) \times C_t \times M(t) \to \Jones_n(\delta)$$ $$(p,\sigma,q) \to C_{p,q}^{\sigma}$$ is an injection onto the $R$-basis of annular diagrams. \qed
\end{proposition}

The picture of the correspondence is as follows, with $n=11$ and $t=3$. If
\begin{center}
$(p, \sigma, q)=($
\begin{tikzpicture}[x=1.5cm,y=-.5cm,baseline=-1.5cm]
\def\wid{1}
\def\cheatWid{\wid*1.05}
\def\hei{0.5}
\def\nodesize{3}
\def\newAng{120}

\draw[e, dashed] (\wid,-0.5*\hei) rectangle (2*\wid,-0.5*\hei);
\draw[e, dashed] (\wid,10.5*\hei) rectangle (2*\wid,10.5*\hei);

\node[v, minimum size=\nodesize] (b1) at (1* \wid,0*\hei) {};
\node[v, minimum size=\nodesize] (b2) at (1* \wid,1*\hei) {};
\node[v, minimum size=\nodesize] (b3) at (1* \wid,2*\hei) {};
\node[v, minimum size=\nodesize] (b4) at (1* \wid,3*\hei) {};
\node[v, minimum size=\nodesize] (b5) at (1* \wid,4*\hei) {};
\node[v, minimum size=\nodesize] (b6) at (1* \wid,5*\hei) {};
\node[v, minimum size=\nodesize] (b7) at (1* \wid,6*\hei) {};
\node[v, minimum size=\nodesize] (b8) at (1* \wid,7*\hei) {};
\node[v, minimum size=\nodesize] (b9) at (1* \wid,8*\hei) {};
\node[v, minimum size=\nodesize] (b10) at (1* \wid,9*\hei) {};
\node[v, minimum size=\nodesize] (b11) at (1* \wid,10*\hei) {};

\node[v, minimum size=\nodesize] (x1) at (2* \wid,3*\hei) {};
\node[v, minimum size=\nodesize] (x2) at (2* \wid,5*\hei) {};
\node[v, minimum size=\nodesize] (x3) at (2* \wid,7*\hei) {};

\draw[e] (b1) to[out = 0, in = 270] (1.1 * \wid,-0.5*\hei);
\draw[e] (b2) to[out = 0, in = 270] (1.2 * \wid,-0.5*\hei);
\draw[e] (b11) to[out = 0, in = 90] (1.1 * \wid,10.5*\hei);
\draw[e] (b10) to[out = 0, in = 90] (1.2 * \wid,10.5*\hei);

\draw[e] (b4) to[out=0, in=0] (b5);
\draw[e] (b6) to[out=0, in=0] (b7);

\draw[e] (b3) to[out=0, in=180] (x1);
\draw[e] (b8) to[out=0, in=180] (x2);
\draw[e] (b9) to[out=0, in=180] (x3);

\end{tikzpicture}
$,$
\quad
\begin{tikzpicture}[x=1.5cm,y=-.5cm,baseline=-1.5cm]
\def\wid{1}
\def\cheatWid{\wid*1.05}
\def\hei{0.5}
\def\nodesize{3}
\def\newAng{120}

\draw[e, dashed] (\wid,2.5*\hei) rectangle (2*\wid,2.5*\hei);
\draw[e, dashed] (\wid,7.5*\hei) rectangle (2*\wid,7.5*\hei);

\node[v, minimum size=\nodesize] (x1) at (1* \wid,3*\hei) {};
\node[v, minimum size=\nodesize] (x2) at (1* \wid,5*\hei) {};
\node[v, minimum size=\nodesize] (x3) at (1* \wid,7*\hei) {};

\node[v, minimum size=\nodesize] (y1) at (2* \wid,3*\hei) {};
\node[v, minimum size=\nodesize] (y2) at (2* \wid,5*\hei) {};
\node[v, minimum size=\nodesize] (y3) at (2* \wid,7*\hei) {};

\draw[e] (x1) to[out=0, in=180] (y3);
\draw[e] (x2) to[out=0, in=\newAng] (1.55 * \wid,7.5*\hei);
\draw[e] (x3) to[out=0, in=\newAng] (1.45 * \wid,7.5*\hei);

\draw[e] (1.55 * \wid,2.5*\hei) to[out=180+\newAng, in=180] (y1);
\draw[e] (1.45 * \wid,2.5*\hei) to[out=180+\newAng, in=180] (y2);
\end{tikzpicture}
$,$
\quad
\begin{tikzpicture}[x=1.5cm,y=-.5cm,baseline=-1.5cm]
\def\wid{1}
\def\cheatWid{\wid*1.05}
\def\hei{0.5}
\def\nodesize{3}
\def\newAng{120}

\draw[e, dashed] (\wid,-0.5*\hei) rectangle (2*\wid,-0.5*\hei);
\draw[e, dashed] (\wid,10.5*\hei) rectangle (2*\wid,10.5*\hei);

\node[v, minimum size=\nodesize] (y1) at (1* \wid,3*\hei) {};
\node[v, minimum size=\nodesize] (y2) at (1* \wid,5*\hei) {};
\node[v, minimum size=\nodesize] (y3) at (1* \wid,7*\hei) {};

\node[v, minimum size=\nodesize] (c1) at (2* \wid,0*\hei) {};
\node[v, minimum size=\nodesize] (c2) at (2* \wid,1*\hei) {};
\node[v, minimum size=\nodesize] (c3) at (2* \wid,2*\hei) {};
\node[v, minimum size=\nodesize] (c4) at (2* \wid,3*\hei) {};
\node[v, minimum size=\nodesize] (c5) at (2* \wid,4*\hei) {};
\node[v, minimum size=\nodesize] (c6) at (2* \wid,5*\hei) {};
\node[v, minimum size=\nodesize] (c7) at (2* \wid,6*\hei) {};
\node[v, minimum size=\nodesize] (c8) at (2* \wid,7*\hei) {};
\node[v, minimum size=\nodesize] (c9) at (2* \wid,8*\hei) {};
\node[v, minimum size=\nodesize] (c10) at (2* \wid,9*\hei) {};
\node[v, minimum size=\nodesize] (c11) at (2* \wid,10*\hei) {};

\draw[e] (y3) to[out=0, in=180] (c7);
\draw[e] (y1) to[out=0, in=180] (c1);
\draw[e] (y2) to[out=0, in=180] (c4);

\draw[e] (c5) to[out=180, in=180] (c6);
\draw[e] (c2) to[out=180, in=180] (c3);
\draw[e] (c8) to[out=180, in=180] (c11);
\draw[e] (c9) to[out=180, in=180] (c10);

\end{tikzpicture}
$) \in M(3) \times C_3 \times M_3$,
\end{center}
(here $\sigma$ is the `rotation' $i \mapsto i+2$ and we have drawn `hanging edges' at defects of $p$ and $q$), then 
\begin{center}
$C^{\sigma}_{p,q}=$
\begin{tikzpicture}[x=1.5cm,y=-.5cm,baseline=-1.5cm]
\def\wid{2}
\def\cheatWid{\wid*1.05}
\def\hei{0.5}
\def\nodesize{3}
\def\newAng{120}

\draw[e, dashed] (\wid,-0.5*\hei) rectangle (2*\wid,-0.5*\hei);
\draw[e, dashed] (\wid,10.5*\hei) rectangle (2*\wid,10.5*\hei);

\node[v, minimum size=\nodesize] (b1) at (1* \wid,0*\hei) {};
\node[v, minimum size=\nodesize] (b2) at (1* \wid,1*\hei) {};
\node[v, minimum size=\nodesize] (b3) at (1* \wid,2*\hei) {};
\node[v, minimum size=\nodesize] (b4) at (1* \wid,3*\hei) {};
\node[v, minimum size=\nodesize] (b5) at (1* \wid,4*\hei) {};
\node[v, minimum size=\nodesize] (b6) at (1* \wid,5*\hei) {};
\node[v, minimum size=\nodesize] (b7) at (1* \wid,6*\hei) {};
\node[v, minimum size=\nodesize] (b8) at (1* \wid,7*\hei) {};
\node[v, minimum size=\nodesize] (b9) at (1* \wid,8*\hei) {};
\node[v, minimum size=\nodesize] (b10) at (1* \wid,9*\hei) {};
\node[v, minimum size=\nodesize] (b11) at (1* \wid,10*\hei) {};

\node[v, minimum size=\nodesize] (c1) at (2* \wid,0*\hei) {};
\node[v, minimum size=\nodesize] (c2) at (2* \wid,1*\hei) {};
\node[v, minimum size=\nodesize] (c3) at (2* \wid,2*\hei) {};
\node[v, minimum size=\nodesize] (c4) at (2* \wid,3*\hei) {};
\node[v, minimum size=\nodesize] (c5) at (2* \wid,4*\hei) {};
\node[v, minimum size=\nodesize] (c6) at (2* \wid,5*\hei) {};
\node[v, minimum size=\nodesize] (c7) at (2* \wid,6*\hei) {};
\node[v, minimum size=\nodesize] (c8) at (2* \wid,7*\hei) {};
\node[v, minimum size=\nodesize] (c9) at (2* \wid,8*\hei) {};
\node[v, minimum size=\nodesize] (c10) at (2* \wid,9*\hei) {};
\node[v, minimum size=\nodesize] (c11) at (2* \wid,10*\hei) {};

\draw[e] (b1) to[out = 0, in = 270] (1.067 * \wid,-0.5*\hei);
\draw[e] (b2) to[out = 0, in = 270] (1.13 * \wid,-0.5*\hei);
\draw[e] (b11) to[out = 0, in = 90] (1.067 * \wid,10.5*\hei);
\draw[e] (b10) to[out = 0, in = 90] (1.13 * \wid,10.5*\hei);

\draw[e] (b4) to[out=0, in=0] (b5);
\draw[e] (b6) to[out=0, in=0] (b7);

\draw[e] (b3) to[out=0, in=180] (c7);
\draw[e] (b8) to[out=0, in=\newAng] (1.55 * \wid,10.5*\hei);
\draw[e] (b9) to[out=0, in=\newAng] (1.45 * \wid,10.5*\hei);

\draw[e] (1.55 * \wid,-0.5*\hei) to[out=180+\newAng, in=180] (c1);
\draw[e] (1.45 * \wid,-0.5*\hei) to[out=180+\newAng, in=180] (c4);

\draw[e] (c5) to[out=180, in=180] (c6);
\draw[e] (c2) to[out=180, in=180] (c3);
\draw[e] (c8) to[out=180, in=180] (c11);
\draw[e] (c9) to[out=180, in=180] (c10);

\end{tikzpicture}
$\in \mathrm{J}_{11}(\delta)$.
\end{center}

Henceforth, we will typically write $C_n \cup C_n'$ for the set of vertices, remembering the ordering given by $C_n \cong \underline{n}$ only when necessary.

Recall that the twosided ideal $I_{\leq n-1}$ is the $R$-span of the diagrams with fewer than the maximal number $n$ of left-to-right connections.

\begin{remark} \label{rmk: Jones trivial module} The Jones algebra $\Jones_n(\delta)$ comes with a natural trivial module, or equivalently an augmentation, which we as usual denote $\t$. This module is a copy of $R$, where annular diagrams with $n$ left-to-right connections act by 1, and other annular diagrams act by 0. With this convention, the retraction $$R C_n \to \Jones_n(\delta) \to R C_n$$ given in the introduction is a retraction of augmented algebras. In particular, $I_{\leq n-1}$ acts trivially on $\t$ and we have an isomorphism $\faktor{\Partition_n(\delta)}{I_{\leq n-1}} \cong R C_n$ of augmented algebras. 
\end{remark}

\begin{definition} For $i \in C_n$, let $J_i$ be the left ideal of $\Jones_n(\delta)$ spanned by diagrams where the vertices $i'$ and $(i+1)'$ on the right are connected by an edge. \end{definition}

\begin{lemma} \label{lem:JCover} The ideals $J_i$ cover $I_{\leq n-1}$. \end{lemma}

\begin{proof} Any diagram $x$ with fewer than $n$ left-to-right connections must have a primed vertex $i'$ not connected to an unprimed vertex. This vertex must therefore be connected to some other vertex $j'$, so by Proposition \ref{prop:JCellular}, and the definition of annular link states (Definition \ref{def:AnnularLinkState}), at least one of the cyclic intervals $(i',j')$ or $(j',i')$ consists entirely of vertices with right-to-right connections. Without loss of generality, suppose that $(i',j')$ consists entirely of vertices with right-to-right connections. Choose such a right-to-right connection, and call its ends $i'_1$ and $j'_1$ (choosing which end receives which name so that $(i'_1,j'_1) \subset (i',j')$). If $(i'_1,j'_1)$ is empty, then $i'_1$ and $j'_1$ must be adjacent. If not, then (as a subset of $(i',j')$) it must consist entirely of vertices with right-to-right connections, and we may choose one of \emph{these} and repeat. At each stage, the cyclic interval becomes smaller, so we eventually reach a connection between two adjacent vertices, say $k'$ and $(k+1)'$, whence $x \in J_k$. \end{proof}

Let $T \subset C_n$. Borrowing terminology from \cite{Sroka}, we will say that $T$ is \emph{innermost} if there do not exist distinct elements $i \neq j$ in $T$ such that $i = j+1$ or $j= i+1$. As in that paper, the point is that the innermost sets are precisely those $T$ for which there exists an annular diagram where $i'$ and $(i+1)'$ are connected for every $i \in T$.
 
\begin{lemma} \label{lem:JFormOfIntersection} Let $T \subset C_n$. The intersection $$\bigcap_{i \in T} J_i$$ is the $R$-span of those partitions which have an edge between vertices $i'$ and $(i+1)'$ whenever $i \in T$. This intersection is nonzero if and only if $T$ is innermost.
\end{lemma}

\begin{proof} The first claim is immediate. For the second claim, note first that if $T$ is not innermost then a partition in $\bigcap_{i \in T} J_i$ must have a part of cardinality at least three, and hence certainly cannot be annular. Conversely, if $T$ is innermost, then consider the collection $q$ of subsets of $C_n$ which consists of pairs $\{i,i+1\}$ for each $i \in T$ and singletons $\{s\}$ for each $s \in C_n$ which is not equal to $i$ or $i+1$ for some $i \in T$. If $T$ is innermost, then $q$ is a partition, and is automatically annular (both of the conditions of Definition \ref{def:AnnularLinkState} holding trivially, since the open cyclic interval $(i,i+1)$ is empty), and Proposition \ref{prop:JCellular} implies there exists a diagram whose right link state is $q$ (for example, $C_{q,q}^{1_{C_t}}$).
\end{proof}

For $T$ a subset of $C_n$, say that the \emph{moral support} of $T$ is the set $\mathrm{MS}(T) = T \cup (T+1) \subset C_n$. For $S$ a subset of $C_n$, and $b \in S$, say that $b$ is \emph{locally (cyclically) minimal} in $S$ if $b \in S$, but $b-1 \not \in S$. Note that $S$ contains a locally minimal element precisely when it is nonempty and proper.

The following lemma is then immediate.

\begin{lemma} \label{lem:existsa} Let $T \subset C_n$. If $\textrm{MS}(T)$ is a nonempty proper subset of $C_n$, then there exists $a \in C_n$ so that $a+2$ is locally minimal in the complement $C_n \setminus \textrm{MS}(T)$. \qed \end{lemma}

\begin{lemma} \label{lem:Jmap} Let $T \subset C_n$ be innermost. Suppose that $a \in T$ is such that $a+2$ is locally minimal in the complement of the moral support $\mathrm{MS}(T)$ (so that in particular $\mathrm{MS}(T)$ must be nonempty and proper). Let $\omega$ be the annular diagram where:
\begin{itemize}
    \item $a+2$ is connected to $a+1$,
    \item $a$ is connected to $(a+2)'$,
    \item $(a+1)'$ is connected to $a'$,
    \item for $i \in (a+2,a) = C_n \setminus [a,a+2]$, $i$ is connected to $i'$.
\end{itemize}
Then $(\bigcap_{i \in T \setminus \{a\} } J_i) \cdot \omega \subset \bigcap_{i \in T} J_i$.
\end{lemma}

The picture is the same as the key one in Sroka's paper \cite{Sroka}. If $a=3$ in $\Jones_{8}$, then:
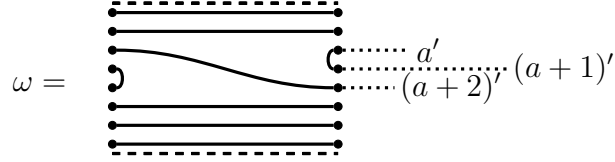
\begin{center}
$\omega = $
\quad
\begin{tikzpicture}[x=1.5cm,y=0.8cm,baseline=1.2cm]

\def\wid{2}
\def\cheatWid{\wid*1.05}
\def\hei{0.5}
\def\nodesize{3}
\def\newAng{120}

\draw[e, dashed] (2*\wid,-0.5*\hei) rectangle (3*\wid,-0.5*\hei);
\draw[e, dashed] (2*\wid,7.5*\hei) rectangle (3*\wid,7.5*\hei);

\node[v, minimum size=\nodesize] (c1) at (2* \wid,0*\hei) {};
\node[v, minimum size=\nodesize] (c2) at (2* \wid,1*\hei) {};
\node[v, minimum size=\nodesize] (c3) at (2* \wid,2*\hei) {};
\node[v, minimum size=\nodesize] (c4) at (2* \wid,3*\hei) {};
\node[v, minimum size=\nodesize] (c5) at (2* \wid,4*\hei) {};
\node[v, minimum size=\nodesize] (c6) at (2* \wid,5*\hei) {};
\node[v, minimum size=\nodesize] (c7) at (2* \wid,6*\hei) {};
\node[v, minimum size=\nodesize] (c8) at (2* \wid,7*\hei) {};

\node[v, minimum size=\nodesize] (d1) at (3* \wid,0*\hei) {};
\node[v, minimum size=\nodesize] (d2) at (3* \wid,1*\hei) {};
\node[v, minimum size=\nodesize] (d3) at (3* \wid,2*\hei) {};
\node[v, minimum size=\nodesize] (d4) at (3* \wid,3*\hei) {};
\node[v, minimum size=\nodesize] (d5) at (3* \wid,4*\hei) {};
\node[v, minimum size=\nodesize] (d6) at (3* \wid,5*\hei) {};
\node[v, minimum size=\nodesize] (d7) at (3* \wid,6*\hei) {};
\node[v, minimum size=\nodesize] (d8) at (3* \wid,7*\hei) {};

\draw[e] (c1) to[out=0, in=180] (d1);
\draw[e] (c2) to[out=0, in=180] (d2);

\draw[e] (c5) to[out=0, in=0] (c4);
\draw[e] (c3) to[out=0, in=180] (d5);
\draw[e] (d3) to[out=180, in=180] (d4);

\draw[e] (c6) to[out=0, in=180] (d6);
\draw[e] (c7) to[out=0, in=180] (d7);
\draw[e] (c8) to[out=0, in=180] (d8);

\draw (3.4* \wid,2*\hei) node{$a'$};
\draw (4* \wid,3*\hei) node{$(a+1)'$};
\draw (3.5* \wid,4*\hei) node{$(a+2)'$};

\draw[e, dotted] (d3) to (3.3* \wid,2*\hei);
\draw[e, dotted] (d4) to (3.75* \wid,3*\hei);
\draw[e, dotted] (d5) to (3.25* \wid,4*\hei);
\end{tikzpicture}
\quad
\end{center}

\begin{proof} Let $\rho$ be a diagram in $\bigcap_{i \in T \setminus \{a\} } J_i$. We must show that $\rho \omega \in \bigcap_{i \in T} J_i$.

First, since $\omega$ has an edge connecting $(a+1)'$ and $a'$, $\rho \omega$ also has an edge connecting these vertices, and hence $\rho \omega$ is in $J_a$.

Suppose now that $j \in T \setminus{a}$. Neither $j=a+1$ nor $j+1=a$ can hold, since this would contradict the assumption that $T$ was innermost. Thus, since $j$ is not in the set $\{a - 1, a, a + 1\}$, and $a+2$ lies in the complement of $T$, so in particular $a+2 \neq j$, we have that $j$ and $j+1$ are both in the open cyclic interval $(a+2,a)$. Thus, $j$ is connected to $j'$ in $\omega$, and likewise $j+1$ is connected to $(j+1)'$. Since $j'$ and $(j+1)'$ are assumed to be connected in $\rho$, it now follows that they are still connected in $\rho \omega$. Thus $\rho \omega$ is in $J_{j}$.

Since $j$ was chosen arbitrarily from $T \setminus{a}$, and we have already established that $\rho \omega \in J_a$, it follows that $\rho \omega $ is contained in the intersection $\bigcap_{i \in T } J_i$, as required.  \end{proof}

\begin{lemma} \label{lem:Jrtrct} Let $T \subset C_n$ be innermost. Suppose that $a \in T$ is such that $a+2$ is locally minimal in the complement of the moral support $\mathrm{MS}(T)$. The map $$ \cdot \omega: \bigcap_{i \in T \setminus \{a\} } J_i \to \bigcap_{i \in T} J_i$$ constructed in Lemma \ref{lem:Jmap} is a retraction of the inclusion of $\bigcap_{i \in T} J_i$ into $\bigcap_{i \in T \setminus \{a\} } J_i$.
\end{lemma}

Note again that the hypotheses imply that $\mathrm{MS}(T)$ must be nonempty and proper.

\begin{proof} Suppose $\rho \in \bigcap_{i \in T} J_i$. We must show $\rho \omega = \rho$.

Since $a \in T$, we know that $\rho \in J_a$, so $a'$ and $(a+1)'$ are connected in $\rho$. First, the product $\rho \omega$ produces no factors of $\delta$: the only left-to-left connection in $\omega$ is the one between $(a+2)$ and $a+1$, so there can only be a loop if $(a+2)'$ and $(a+1)'$ are connected in $\rho$. They cannot be, because $(a+1)'$ is connected to $a'$ in $\rho$. It follows that $\rho \omega$ is equal to the underlying partition $\rho * \omega$ (c.f. Definition \ref{def:partitionAlg}) - that is, the multiplication does not produce any $\delta$ factors - and we will now argue that this partition is equal to $\rho$.

Since $a'$ and $(a+1)'$ are connected in $\rho$, $(a+2)'$ is connected to $(a+2)''$ in the composed partition $\rho \circ \omega'$ (c.f. Definition \ref{def:partitionAlg}). More generally, this means that $i'$ is connected to $i''$ for $i$ not equal to $a$ or $a+1$. It follows that for such $i$, $i'$ is connected to the same vertex in $\rho * \omega$ as it was in $\rho$, and since $(a+1)'$ and $a'$ are connected in $\omega$, they are also connected in $\rho * \omega$. This establishes that $\rho* \omega$ has all of the right-to-right and left-to-right connections from $\rho$, and it is automatic that the product retains all left-to-left connections from $\rho$. Thus, $\rho * \omega = \rho$, as required. \end{proof}

\begin{lemma} \label{lem:JIdempotents} Suppose that $T \subset C_n$ is innermost. Unless $n$ is even and $T$ consists either of all of the odd or all of the even elements of $C_n$, the ideal $$\bigcap_{i \in T} J_i$$ is a principal ideal generated by an idempotent. If $\delta$ is invertible, then $\bigcap_{i \in T} J_i$ is principal idempotent for any innermost $T$. \end{lemma}

\begin{proof} If $T$ is empty then $\bigcap_{i \in T} J_i = \mathrm{J}_n(\delta)$ and we may take the idempotent to be the identity diagram. If $T$ is nonempty innermost, then the moral support $\textrm{MS}(T)$ is a proper subset of $C_n$ unless $n$ is even and $T$ consists either of all of the odd or all of the even elements of $C_n$. The result then follows by combining Lemma \ref{lem:existsa}, Lemma \ref{lem:Jrtrct} and Lemma \ref{lem:retractIsIdemp}. The case of invertible $\delta$ is left to the reader. \end{proof}

We are now ready to prove the main result on $\Jones_n(\delta)$. Recall from the introduction and from Remark \ref{rmk: Jones trivial module} that we have an isomorphism of augmented algebras $$\faktor{\Jones_n(\delta)}{I_{\leq n-1}} \cong R C_n,$$ and that we use $\t$ to denote the trivial module given by either the augmentation of $\Jones_n(\delta)$ or that of $R C_n$.

\begin{proof}[Proof of Theorems \ref{thm:Jones} and \ref{thm:JonesSroka}] By Lemma \ref{lem:JCover}, the ideals $J_i$ form an $R$-free cover of $I_{\leq n-1}$. By Lemma \ref{lem:JIdempotents}, an intersection of at most $\frac{n}{2}-1$ ideals from among the $J_i$ is either zero or principal idempotent, so this is indeed a principal idempotent cover of height $\frac{n}{2}-1$. Since $I_{\leq n-1}$ acts trivially on $\t$, Theorem \ref{thm:Jones} then follows by Theorem \ref{thm:main}. 

If $n$ is odd or $\delta$ is invertible, then Lemma \ref{lem:JIdempotents} gives that \emph{any} intersection of ideals from among the $J_i$ is either zero or principal idempotent, so Theorem \ref{thm:JonesSroka} follows by applying Theorem \ref{thm:main} for arbitrarily large finite $h$. \end{proof}

\printbibliography

\end{document}